\documentclass[a4paper,10pt,twoside]{article}
\usepackage[latin1]{inputenc} 
\usepackage{graphicx}
\usepackage[pdftex,bookmarks]{hyperref}
\usepackage{amsmath}    
\usepackage{amssymb}
\usepackage{amsmath,amsthm}
\usepackage{fancyhdr}
\usepackage{textcomp}
\usepackage{color}
\usepackage[all]{xy}
\usepackage{slashed}
\usepackage{latexsym}

\usepackage{mathrsfs}

\usepackage{fancyhdr}
\fancypagestyle{plain}{
	\fancyhead{}
	
}

\numberwithin{equation}{section}
\newtheorem{theorem}{Theorem}[section]
\newtheorem{lemma}[theorem]{Lemma}
\newtheorem{proposition}[theorem]{Proposition}
\newtheorem{corollary}[theorem]{Corollary}

\theoremstyle{definition}
\newtheorem{definition}[theorem]{Definition}

\newtheorem{remark}[theorem]{Remark}
\theoremstyle{plain}

\def \red{\textcolor{red}}

\def\BB{\ensuremath\mathbf{B}}

\def\CC{\ensuremath{\mathbb{C}}}
\def\NN{\ensuremath{\mathbb{N}}}

\def\RR{\ensuremath{\mathbb{R}}}
\def\ZZ{\ensuremath{\mathbb{Z}}}
\date{}
\begin{document}
	
	\title{Mapping the surgery exact sequence for topological manifolds to analysis}
	
	\author
	{Vito Felice Zenobi}
	\maketitle
	\date{}
	
	\begin{abstract}
		In this paper we prove the existence of a natural mapping from the 
		surgery exact sequence for topological manifolds to the analytic surgery exact sequence of N. Higson and J. Roe.
		This generalizes the fundamental result of Higson and Roe, but in the treatment given by Piazza and Schick, from smooth manifolds to topological manifolds. Crucial to our treatment is the Lipschitz signature operator of Teleman. 
		
		We also give a generalization to the equivariant setting of the product defined by Siegel in his Ph.D. thesis. Geometric applications are given to stability results for rho classes. We also obtain a proof of the
		APS delocalised index theorem on odd dimensional manifolds, both for the spin Dirac operator and
		the signature operator, thus extending to odd dimensions the results of Piazza and Schick.
		Consequently, we are able to discuss the mapping of the surgery sequence in all dimensions.
	\end{abstract}
	
	\tableofcontents

	\section{Introduction}
	
	Let $M$ be a $n$-dimensional topological manifold, with $\Gamma=\pi_1(M)$ and let $\widetilde{M}\to M$ be its universal covering.
	We assume $n$ greater than $5$ and, initially, odd.
	
	In \cite{Sull} Sullivan proves that there always exists a Lipschitz manifold structure on $M$ and that it is unique up to a bi-Lipschitz homeomorphism isotopic to the identity.  
	In \cite{Tel1,Tel2} Teleman studies index theory in  the Lipschitz context and  in \cite{Hil1} Hilsum develops it
	in the framework of unbounded Kasparov theory. In particular there is a signature operator arising from the Lipschitz structure
	and this operator determines a well defined class in the K-homology of $M$.
	
	Thanks to these results it is possible to extend the work by Piazza and Schick \cite{PS2} (that follows the one by Higson and Roe \cite{HigRoeI,HigRoeII,HigRoeIII}) from the smooth to the topological category. Let us recall that in \cite{PS2} Piazza and Schick built a natural transformation
	\[
	\xymatrix{L_{n+1}(\ZZ\Gamma)\ar[r]\ar[d]^{\mathrm{Ind}_\Gamma} & \mathcal{S}(M)\ar[r]\ar[d]^\varrho & \mathcal{N}(M)\ar[r]\ar[d]^\beta & L_{n}(\ZZ\Gamma)\ar[d]^{\mathrm{Ind}_\Gamma} \\
		K_{n+1}(C^*_r(\Gamma))\ar[r] & K_{n+1}(D^*(\widetilde{M})^\Gamma)\ar[r] & K_n(M)\ar[r] & K_{n}(C^*_r(\Gamma))}
	\]
	from the surgery exact sequence for smooth manifolds to the analytic surgery exact sequence of Higson and Roe,
	using tools and methods in coarse index theory.
	
	In this paper we  check that this mapping also
	exists  for the surgery sequence for topological manifolds. To this aim we will use as key tool the Lipschitz structure given by Sullivan theorem \cite{Sull}.
	In particular we prove that the key results by Wahl, Piazza and Schick, have a true abstract and K-theoretical meaning,
	that does not depend on the smooth structure and the pseudodifferential calculus.
	
	One significant difference between the smooth SES and the topological SES is that  the second one is an exact sequence of groups, whereas the first one is not. In this paper we deal with the mapping at the set level:  to prove that the diagram is commutative as a diagram of groups, the main difficulty is that the group structure of the topological structure set is rather quite hard to handle. The following question is wide open:
	\begin{itemize}
		\item
		is the map
		$\varrho\colon \mathcal{S}(M)\to K_{n+1}(D^*(\widetilde{M})^\Gamma)$  a homomorphism of groups?
	\end{itemize} 
	
	A positive answer to this question would have direct consequences to
	Conjecture 3.8 in  \cite{WY}, using the methods in developed in \cite{XY2}.

	In the second part of the paper we give another realization of the group $K_*(D^*(\widetilde{M})^\Gamma)$ in terms of the mapping cone of Kasparov bimodules, in order to generalize to the equivariant setting
	a product formula proved by Siegel in his Ph.D. thesis    \cite{Sieg}.
	This product allows us to  give stability results for $\varrho$-invariants and to prove the delocalized APS index theorem of Piazza and Schick (\cite{PS}) in the odd dimensional case.
	This last result leads to define a natural mapping from the SES to the analytic SES of Higson and Roe when $\dim(M)=n$ is even, in both the smooth and the topological setting. 
	
	With the same method one can extendd the construction in \cite{PS}, about the Stolz exact sequence, to the case of even dimensional manifolds, but this was already proven in \cite{XY}.
	
	We refer the reader to \cite{PS2} for a more detailed overview of the problem in the smooth setting.
	
	\subsection*{Acknowledgements.}  I am very thankful to my advisors Paolo Piazza and Georges Skandalis for their support and teachings. Moreover I am glad to thank Thomas Schick and Rudolf Zeidler for pointing out some imprecisions in the first version of the paper. Finally I am grateful to the anonymous referee for his careful reading of the paper. 
	
	\section{Signature operator on Lipschitz manifolds}
	
	We start recalling  fundamental results on Lipschitz manifolds. For further details we refer to \cite{Tel1,Tel2,Hil1,Sull,TV}.
	
	\begin{definition}
		A Lipschitz atlas on a topological manifold  $M$ is an atlas such that the map $\varphi\circ\psi^{-1}$ is a Lipschitz homeomorphism for any
		two charts $\varphi\colon U\to \RR^n$ and $\psi\colon V\to \RR^n$. 
		By definition a Lipschitz manifold structure on $M$ is a maximal Lipschitz atlas.
	\end{definition}
	
	\begin{theorem}[\cite{Sull}]\label{sull}
		Any topological manifold of dimension $n\neq4$ has a Lipschitz atlas of coordinates.
		For any two such structures $L_1$ and $L_2$, there exists a Lipschitz homeomorphism
		$h\colon L_1\to L_2$ isotopic to the identity.
	\end{theorem}
	
	\begin{theorem}[\cite{Tel1,Hil1}]\label{lip}
		Let $M$ be a closed oriented
		Lipschitz manifold of even dimension. Then from the complex of $L^2$-differential
		forms on $M$ (with respect to some choice of a Lipschitz Riemannian metric
		$g$) one obtains a signature operator $D_g$ which is closed and self-adjoint. Therefore $D_g$  
		determines a class $[D]$ in $K_0(M)\simeq KK(C(M),\CC)$ which is independent of the choice of the metric $g$.
		The image of $[D]$ in $K_0(pt)\simeq KK(\CC,\CC)$ (i.e., the index of $D_g$) is the usual signature of the manifold.
	\end{theorem}
	
	In \cite{Hil2} Hilsum proves that the signature operator gives a KK-class as above  in the case of non compact manifolds too, 
	provided the manifold $M$ is endowed with a metric $g$
	such that  it is metrically complete with respect to the induced structure of metric space.
	Moreover he showed a result on the finite propagation speed for  solutions of the wave equation.
	
	\begin{theorem}[Hilsum]\label{Hil}
		Let $M$ be an oriented Lipschitz manifold with a Riemannian structure, such that the manifold is complete as metric space.
		Let $\mbox{d}$ be the associated distance function and let $D$ be the associated signature operator.
		For all $t\in \RR$, we have that:
		\[
		\mathrm{supp}(e^{itD})\subset\{(x,y)\in M\times M\, |\, d(x,y)\leq t\}.
		\]
		For $f\in\mathcal{S}(\RR)$ such that $\mathrm{supp}(\hat{f})\subset[-a,a]$, with $a>0$, we have that:
		\[
		\mathrm{supp}(f(D))\subset\{(x,y)\in M\times M\, |\, d(x,y)\leq a\}.
		\]
	\end{theorem}
	This theorem will be key in the coarse geometrical setting.
	
	\section{The $\varrho$ classes}

	We refer the reader to  \cite[sect.1]{HigRoeIII} and  \cite[sect.1]{PS} 
	for notations about coarse geometry and coarse algebras.
	
	Let $N$ be an oriented topological manifold of dimension $n\geq5$; an element of the   topological structure set  of $N$ is given by an orientation preserving homotopy 
	equivalence  $f\colon M\to N$. Two  different homotopy equivalences $f\colon M\to N$ and 
	$f'\colon M'\to N$ are equivalent if there is a $h$-cobordism $W$ between them and 
	a homotopy equivalence $F\colon W \to N\times[0,1]$, such that $F_{|M}=f$ and $F_{|M'}=f'$.
	\begin{definition}
		The topological structure set $\mathcal{S}^{TOP}(N)$ of $N$ is defined as  
		the set of the $h$-cobordism classes of oriented homotopy equivalences.
	\end{definition}

	Given a class $[f\colon M\to N]$, we set $Z=M\cup-N$. Let $\Gamma$ be the fundamental group of $N$.
	The universal covering $\widetilde{N}\to N$ is induced by a map $u\colon N\to B\Gamma$, namely $\widetilde{N}=u^*(E\Gamma)$,
	where $B\Gamma$ is the classifying space of $\Gamma$ and $E\Gamma$ is its universal covering.
	Let $\widetilde{M}$ be the $\Gamma$-Galois covering induced by $u_M:=u\circ f$, then we get a $\Gamma$-Galois covering 
	$\widetilde{Z}=\widetilde{M}\cup-\widetilde{N}$ on $Z$.
	Let  $\mathcal{F}=\widetilde{Z}\times_\Gamma C^*_r(\Gamma)$ be the associated  Mischenko bundle.

	Now, starting from a Lipschitz structure on $Z$ given by Theorem \ref{sull},
	consider the $L^2$-forms complex $L^2(Z,\Lambda_\CC(Z))$, see  \cite[Section 2]{Hil1}.
	
	We get a differential $d_Z$ and
	an involution $\tau_Z$;  $\tau_Z$ is the operator  
	$\omega\mapsto i^{p(p-1)+\frac{n}{2}}*\omega$ on  forms of degree $p$.
	Like in the classical Hodge theory we can define the Lipschitz signature operator (with coefficients) as 
	\[
	D_Z=(d_Z-\tau_Z d_Z \tau_Z)
	\]
	if $n$ is even and
	\[
	D_Z=(\tau_Z d_Z+d_Z \tau_Z)
	\]
	if $n$ is odd.
	
	Like in \cite{Hil1}, we have that $\left(L^2(Z,\Lambda_\CC(Z)),\mu,  D_Z\right)$ defines  
	an unbounded class 
	$[D_Z]\in KK(C(Z),\CC)$,
	where $\mu$ is the representation that associates the multiplication operator $\mu_f$ to a function $f$.

	\subsection{The perturbed signature operator}
	
	We want to associate a class in the K-group $K_*(D^*(\tilde{M})^\Gamma)$ to a homotopy equivalence $f\colon ~M\to ~N$  and show that this mapping is well defined
	on the $h$-cobordism classes.
	
	The key result for what follows is the homotopy invariance of the index class of the signature
	operator for compact oriented smooth manifolds, proved by M. Hilsum and G. Skandalis in \cite{HilSk}.
	Remember that, in the equivariant setting, this class is given by
	\[
	\mathrm{Ind}_\Gamma(D_Z)=[\mathcal{F}]\otimes_{C(X)\otimes C^*_r(\Gamma)}[D_Z]\in KK(\CC,C_r^*(\Gamma)),
	\]
	where $[\mathcal{F}]$ is the class of Mishchenko bundle in $KK(\CC,C(Z)\otimes C^*_r(\Gamma))$.
	
	\begin{theorem}[Hilsum-Skandalis]\label{hs}
		Let $f\colon M\to N$ be a homotopy equivalence.
		Then the class 
		$\mathrm{Ind}_\Gamma(D_Z)\in KK(\CC,C^*_r(\Gamma))$
		vanishes.
	\end{theorem}
	
	In remark \cite[p.95]{HilSk} the authors observe that all  arguments can be applied to the Lipschitz case: we can easily check that the objects do not need be smooth. 
	
	\begin{remark}\label{pert}
		Of particular interest to us is a byproduct of the proof of Theorem 3.1,
		namely the construction of a homotopy $\mathcal{D}_\alpha$ between the signature operator $\mathcal{D}_Z=\mathcal{D}_0$ and an invertible operator $\mathcal{D}_1$.
		This is the reason why the index class $\mathrm{Ind}_\Gamma (D_Z)$ vanishes.
		Here $\mathcal{D}_Z$ is the signature operator twisted by the Mishchenko bundle.
		Moreover we point out that the perturbation creates a gap in its  spectrum near $0$.
	\end{remark}
	
	\begin{proposition}
		The difference $\mathcal{D}_0-	\mathcal{D}_1$ is a C*-module compact operator on 
		$L^2(Z,\Lambda_\CC(Z)\otimes~\mathcal{F})$ both in the smooth and in the Lipschitz case. 
	\end{proposition}
	
	\begin{proof}
		\item The proof of \cite[Theorem 3.3]{HilSk} is based on the construction of an operator $T_{p,v}$, that plays the role of the pull-back of forms.
		
		Let us take the following data:
		\begin{itemize}
			\item a submersion $p\colon M\times B^k\to N$, where $B^k$ is the  unit open disk of $\RR^k$;
			\item a smooth $k$-form $v$ with compact support on $B^k$, such that $\int_{B^k}v=1$. Put then  $\omega=~p_{B^k}^*(v)$.
		\end{itemize}
		Then $p^*\colon L^2(N,\Lambda_\CC(N)\otimes\mathcal{F}_N)\to L^2(M\times B^k,\Lambda_\CC(M\times B^k)\otimes p^*\mathcal{F}_N)$ is a bounded operator and $T_{p,v}$ is defined as the operator $\xi\mapsto q_*(\omega\wedge p^*(\xi))$.
		Consider the following commutative diagram
		\begin{equation}\label{submer}
			\xymatrix{& M\times B^k\ar[rd]^p\ar[ld]_{q}\ar[d]_{t}& \\
				M & M\times N\ar[l]^{p_M}\ar[r]_{p_N}& N}
		\end{equation}
		where $t=id_M\times p$. We get that, for $\xi\in L^2(N,\Lambda_\CC(N)\otimes\mathcal{F_N})$
		\begin{equation*}
			\begin{split}
				T_{p,v}(\xi)&=q_*(\omega\wedge p^*(\xi))=\\
				&=(p_M)_*t_*\left(\omega\wedge(t)^*p_N^*(\xi)\right)=\\
				&=(p_M)_*\left(t_*\omega\wedge p_N^*(\xi)\right).
			\end{split}
		\end{equation*}
		Notice that $(p_M)_*$ is nothing but the integration over $N$. Assume that $k$ and $p$ are chosen so that $t$ is a submersion.
		If we denote  the form $t_*\omega$ on $M\times N$ with $k(y,x)$, it turns out that $T_{p,v}(\xi)=\int_N k(x,y)\xi(x)$ is an integral operator with smooth kernel and consequently a smoothing operator.

		The operator $Y$ in \cite[Lemma 2.1(c)]{HilSk}, such that $1+T'_{p,v}\circ T_{p,v}=d_N\circ Y+Y\circ d_N$, is bounded of order $-1$ (see the proof of \cite[Lemma 2.2]{Wahl} for an explicit expression of $Y$).
		
		Now we can follow word by word the proof of \cite[Lemma 9.14]{PS3}, using the conventions in \cite[Section 3]{Wahl}.
		For simplicity let us consider the odd case. The perturbed signature  operator is then
		given by 
		\[
		\mathcal{D}_t=-iU_t(d_t\circ S_t+ S_t\circ d_t)\circ U_t^{-1}
		\] 
		where
		\[
		d_t=\begin{pmatrix}
		d_M & t T'_{p,v}\\
		0 &-d_N
		\end{pmatrix},\,\mbox{  }\, S_t=\mathrm{sign}\left(\tau_Z\circ L_t\right)\,\mbox{and  }\, U_t=|\tau_Z\circ L_t|^{\frac{1}{2}},
		\]
		with 
		\[
		L_t=\begin{pmatrix}
		1+T'_{p,v}\circ T_{p,v} &  (1-it\gamma\circ Y)\circ T'_{p,v}\\
		T_{p,v}\circ(1+it\gamma\circ Y) &1
		\end{pmatrix}
		\]
		One can easily see that $L_t=1+ H_t$, with $H_t$ smoothing.
		Moreover one has that  $|\tau_Z\circ L_t|=\sqrt{L_t^*L_t}=\sqrt{1+R_t}$, with $R_t$ smoothing.
		Observe that $0<L_t^*L_t=1+R_t$ implies that $R_t>-1$.
		It follows that $\sqrt{L_t^*L_t}-1=f(R_t)$, where
		$f(x)=\sqrt{1+x}-1$ is holomorphic on the spectrum of
		$R_t$ (-1 is a branch point for $f$). 
		Since $f(0)=0$, we have that $f(z)=az+zh(z)z$, where $h$ is a holomorphic function.
		
		Let us point out that if $S_0$ and $S_1$ are smoothing operators and $T$ is a bounded operator, then
		$S_0\circ T\circ S_1$ is smoothing.
		Then we immediately get that  $F_t:=|\tau_Z\circ L_t|-1=f(R_t)$ is smoothing. 
		With the same argument one can prove that $U_t=1+H_t$ with $H_t$ smoothing.
		
		By \cite[Lemma A.12]{PS3}, $|\tau_Z\circ L_t|^{-1}=1+F'_t$ and
		$U_t=1+H'_t$ with $F'_t$ and $H'_t$ smoothing. 
		Then one obtains that, $S_t=\tau_Z+G_t$ and $d_t=d_Z+ E_t$, where $G_t$ and $E_t$ are smoothing operators.
		
		Consequently one has that
		\[\mathcal{D}_t= -i(1+H_t)\left((d+E_t)\circ(\tau_Z+G_t)+(\tau_Z+G_t)\circ(d+E_t)\right)\circ(1+H'_t)\]
		is equal to $\mathcal{D}+ C_f$ with $C_f$ a compact operator.
		
		\item Now we have to prove that the Lipschitz Hilsum-Skandalis perturbation is bounded.
		In the smooth case we tackled the problem geometrically, here we try with a more analytical approach.
		An operator of order $-n$ is a bounded operator between $H^s(Z,E)$ and $H^{s+n}(Z,E)$, the Sobolev sections of $E$ of order $s$ and $s+n$, for any $s$. An operator is regularizing if it is of order $-\infty$.
		Equivalently one can say that an operator $T$ is regularizing (of order $-\infty$) if $D^n\circ T\circ D^m$ is a bounded operator on $L^2$-section for any $m,n\in\ZZ$.
		
		By \cite[Proposition 5.6]{Hil1} we know that the signature operator has compact resolvent, therefore its spectrum is a countable and discrete subset $\{\lambda_n\}_{n\in\NN}$ of $\RR$ such that  $\lim_{n\longrightarrow\infty}\lambda_n^2=+\infty$.
		
		Now let $\psi\in C^\infty_0(\RR)$ be a rapidly decreasing  even function such that $\psi(0)=1$. 
		Since $\psi$ is even, it turns out that $\psi(d_{N}+d_{N}^*)$ maps even/odd degree
		forms to even/odd degree forms and it is a Hilbert-Schmidt  operator: the proof of the first statement of \cite[Prop. 5.31]{Roe} works putting `Hilbert-Schmidt' instead of `smoothing'. Let us denote its kernel by $k(x,y)\in L^2(N\times N,\mathrm{End}(\Lambda_\CC(N)\otimes\mathcal{F}_N))$. 
		
		Define the compact operator $T_f\colon L^2(N,\Lambda_\CC(N)\otimes\mathcal{F}_N)\to L^2(M,\Lambda_\CC(M)\otimes f^*\mathcal{F}_N) $
		as the integral operator with kernel $\bar{q}_* (\omega\wedge\bar{p}^*k)\in L^2(M\times N,\mathrm{Hom}(\Lambda_\CC(N)\otimes\mathcal{F}_N,\Lambda_\CC(M)\otimes f^*\mathcal{F}_N))$, where $\bar{f}=f\times \mathrm{id}_N$ for $f$ equal to $p $ and $q$ as in diagram \eqref{submer}.
		
		This operator satisfies the hypothesis of \cite[Lemma 2.1]{HilSk}. Indeed, because of our choice of $\psi$, we have that $1-\psi(x)=x\cdot \varphi(x)$, where $\varphi$ is a rapidly decreasing odd function. Moreover
		$d_N^*\circ\varphi(d_N+d_N^*)=\varphi(d_N+d_N^*)\circ d_N$, since $\varphi$ is odd.
		Then we get the following formula
		\[
		1-\psi(d_N+d_N^*)=d_N\circ \varphi(d_N+d_N^*)+\varphi(d_N+d_N^*)\circ d_N
		\]
		and by construction $T_f'T_f=\psi(d_N+d_N^*)'\psi(d_N+d_N^*)$.
		Now it is easy to check that there exists an operator $Y\in \mathbb{B}(L^2(N,\Lambda_\CC(N)\otimes\mathcal{F}_N))$ such that $Y(\mathrm{dom}(d_N))\subset \mathrm{dom}(d_N)$ and that $1-T'_f\circ T_f=d_N\circ Y+Y\circ d_N$:
		\begin{equation*}
			\begin{split}
				1-T'_f\circ T_f&=1-\psi'\circ \psi=\\
				&=1-(1-d\circ \varphi-\varphi\circ d)'\circ(1-d\circ \varphi-\varphi\circ d)=\\
				&=1-(1-d\circ\varphi-\varphi\circ d-\varphi'd'-d'\circ\varphi'+\\
				&+d\circ\varphi\circ\varphi'\circ d+d\circ\varphi\circ d'\circ\varphi'+\varphi\circ d\circ\varphi'\circ d'+\varphi\circ d\circ d'\circ\varphi')=\\
				&=d\circ\varphi+\varphi\circ d-\varphi'd-d\circ\varphi'=\\
				&=d\circ Y+ Y\circ d,
			\end{split}
		\end{equation*}
		with $Y= \varphi(d_N+d_N^*)-\varphi(d_N+d_N^*)'$. We simplified some notations and we denoted $d_N$ by $d$, $\varphi(d_N+d^*_N)$ by $\varphi$ and the same for $\psi$. Moreover we used the following facts: $d_N'=-d_N$,  $\varphi\circ d=d^*\circ\varphi$ and $\varphi'\circ d'=(d')^*\circ\varphi'$.

		It is not difficult to check that the operator $T_f$ is a regularizing operator (and hence a compact operator), therefore the image of $T_f$ is in the domain of the Lipschitz signature operator.
		
		Then the boundedness of the Lipschitz Hilsum-Skandalis perturbation follows as in the smooth case.
		The only thing we have to care about is the dependence of this construction on the choice of the metric on $N$. In particular we have to check that $\psi(d_N+d_N^*)$ is Hilbert-Schmidt no matter which metric we use to take the adjoint.
		
		If we have two different metrics $g_0$ and $g_1$ on $N$, then by \cite[Lemma 5.1]{Hil1} we can complete the $Lip(N)$-module $Lip(N,\Lambda_\CC(N)\otimes\mathcal{F}_N)$ with respect to the two metrics and we obtain two isomorphic $C(N)$-Hilbert modules with compatible metrics:
		\[
		K^{-1}||\cdot||_1\leq||\cdot||_0\leq K||\cdot||_1 \,\, \exists K\in \RR^+\setminus{\{0\}}.
		\] 
		Then by the Minmax Theorem $|\lambda_n^0|\leq K^2|\lambda_n^1|$, where for any $n\in \NN$, $\lambda_n^i$ is the $n$-th eigenvalue of $d+d^*_i$.
		
		So it is easy to check that if $\psi$ is a rapidly decreasing function on the spectrum of $d+d^*_0$, it is rapidly decreasing on the spectrum of $d+d^*_1$ too. Therefore $\psi(d+d^*)$ is Hilbert-Schmidt independently of the metric we choice. 
		
	\end{proof}
	
	\begin{definition}
		Let $f\colon M\to N$ be a homotopy equivalence and  $Z=M\cup -N$.
		Denote by $\mathcal{C}_f$ the  perturbation of $\mathcal{D}_Z$ arising in the previous remark and call it a trivializing perturbation. 
		Note that it depends on the homotopy equivalence $f$.
	\end{definition}
	
	We recall from \cite{PS2} that there is an isomorphism of C*-algebras
	\[
	\mathbb{K}(L^2(Z,\Lambda_\CC(Z)\otimes\mathcal{F}))\simeq C^*(\widetilde{Z})^\Gamma.
	\]
	By \cite[Proposition 2.1]{Lance}, the above isomorphism gives an isomorphism at \red{the} level of multiplier algebras
	\begin{equation}\label{iso}
		\mathbb{B}(L^2(Z,\Lambda_\CC(Z)\otimes\mathcal{F}))\simeq \mathcal{M}(C^*(\widetilde{Z})^\Gamma).
	\end{equation}
	This isomorphism is given by the map $L_\pi$ defined in \cite[Section 2.2.1]{PS2}.
	Hence we can go from the Mishchenko bundle setting to the covering one.
	From now on $C_f$ will be the element in $\mathcal{M}(C^*(\widetilde{Z})^\Gamma)$  associated to 
	$\mathcal{C}_f\in \mathbb{B}(L^2(Z,\Lambda_\CC(Z)\otimes\mathcal{F}))$ through the map $L_\pi$. Moreover $\widetilde{D}_Z$ will indicate the operator on the covering induced  by the signature without coefficients in the Mishchenko bundle.

	\begin{remark}\label{fp}
		
		Consider a chopping function $\psi\in C_b(\RR)$ with compactly supported Fourier transform.  Thanks to Theorem \ref{Hil}
		we can prove that the functional calculus through $\psi$
		of the  operator  $\widetilde{D}_Z$ is an operator of finite propagation.
		The pseudolocality of $\widetilde{D}_Z$ comes from \cite[6.1]{Hil1}. Hence
		$\psi(\widetilde{D}_Z) \in D^*(\widetilde{Z})^\Gamma$.
	\end{remark}
	
	\begin{proposition}
		The difference between $\psi(\widetilde{D}_Z)$ and   $\psi(\widetilde{D}_Z+C_f)$ belongs to $C^*(\widetilde{Z})^\Gamma$.  
	\end{proposition}
	
	\begin{proof}
		Moving to the Mishchenko bundle setting through \eqref{iso}, we should prove that
		the difference $\psi(\mathcal{D}_Z)-\psi(\mathcal{D}_Z+\mathcal{C}_f)$ belongs to $\mathbb{K}(L^2(Z,\Lambda_\CC(Z)\otimes\mathcal{F}))$.
		If $\psi_1(t)=t(1+t^2)^{-\frac{1}{2}}$, by \cite[Proposition 2.2]{BJ} we have that $[\psi_1(\mathcal{D}_Z),a]$ belongs to the algebra of compact C*-module operators.
		Therefore if we consider the matrices $\left[\begin{smallmatrix}
		\mathcal{D}_Z& 0\\ 0 &\mathcal{D}_Z+\mathcal{C}_f
		\end{smallmatrix}\right]$ and $\left[\begin{smallmatrix}
		0& 1\\ 1 &0
		\end{smallmatrix}\right]$, their bracket consists in $\left[\begin{smallmatrix}
		0 & -\mathcal{C}_f\\ \mathcal{C}_f &0 
		\end{smallmatrix}\right]$, that is known to be bounded.
		Then, after applying the functional calculus through $\psi_1$, we deduce  that the matrix components in the bracket
		\[\pm(\psi_1(\mathcal{D}_Z)-\psi_1(\mathcal{D}_Z+\mathcal{C}_f))\] are compact.
		
		Now notice that two different chopping functions differ  by a function in $C_0(\RR)$.
		Taking into account the correspondence stated in \eqref{iso}, we have  that the resolvent of $\mathcal{D}_Z$, 
		given by $(i+\mathcal{D}_Z)^{-1}$,
		is compact (see \cite[Proposition 5.6]{Hil1}) and
		since $\phi(t)=(i+t)^{-1}$ generates $C_0(\RR)$, the functional calculus of $D_Z$  through a function in $C_0(\RR)$ gives a compact operator.
		Then if $\psi'$ is any chopping function, it turns out that
		\[
		\psi(\mathcal{D}_Z)-\psi(\mathcal{D}_Z+\mathcal{C}_f)=\psi_1(\mathcal{D}_Z)-\psi_1(\mathcal{D}_Z+\mathcal{C}_f)+\mbox{compacts operators}
		\] 
		and the right-hand side term is compact.
		
	\end{proof}
	
	\begin{corollary} \label{Cor 3.1}
		The operator $\chi(\tilde{D}_Z+C_f)$,
		with  $\chi(x)=\frac{x}{|x|}$, is a bounded involution in $D^*(\widetilde{Z})^\Gamma$.
	\end{corollary}

	Thanks to Corollary \ref{Cor 3.1} we can define  a class  by setting
	\[\varrho(\widetilde{D}_Z+C_f)=\left[\frac{1}{2}(1+\chi(\widetilde{D}_Z+C_f))\right]\in K_0(D^*(\widetilde{Z})^\Gamma).\]
	Now consider the map $\varphi\colon Z\to M$ such that $\varphi_{|N}=f$ and $\varphi_{|-M}=-\mathrm{Id}_M$; we  can clearly see that $\varphi$ is covered by 
	a $\Gamma$-equivariant map $\widetilde{\varphi}\colon \widetilde{Z}\to \widetilde{M}$.

	\begin{definition}\label{rho}
		Let $f\colon M\to N$ be a homotopy equivalence between two compact oriented Lipschitz manifolds.
		Consider $Z=M\cup-N$ and its covering $\widetilde{Z}$ 
		associated, as above, to a classifying map $u\colon Z\to B\Gamma$. Let $\widetilde{D}_Z$ be the Lipschitz signature operator 
		and let $C_f$ be the trivializing perturbation  associated to $f$. We define
		\[
		\varrho (f\colon M\to N):= \widetilde{\varphi}_*\left[\frac{1}{2}(1+\chi(\widetilde{D}_Z+C_f))\right]\in K_0(D^*(\widetilde{M})^\Gamma)
		\]
		and
		\[
		\varrho_\Gamma(f\colon M\to N)=u_*\varrho (f\colon M\to N)\in K_0(D^*_\Gamma).
		\]
		
	\end{definition}
	
	\begin{proposition}\label{dep}
		The $\varrho$-class does not depend on the choice of the Lipschitz structure.
	\end{proposition} 
	\begin{proof}
		The second part of Theorem \ref{sull} can be formulated as follows:
		let $\mathcal{L}_1$ and $\mathcal{L}_2$ be two different Lipschitz structures on $Z$, then there exists a bi-Lipschitz homeomorphism
		$\phi\colon Z\to Z$, isotopic to the identity through a path $\phi^t$ and such that $\phi^*(\mathcal{L}_2)=\mathcal{L}_1$, where $\phi^*\colon C(Z)\to C(Z)$ is the induced *-homomorphism.
		Because of the functoriality of Teleman's construction we know that $\phi_*(\varrho_1)=\varrho_2$, where $\varrho_1$ and $\varrho_2$ are the invariants associated to two different Lipschitz structures.
		The isotopy $\phi^t$ induces a paths of *-isomorphisms 
		$\phi_*^t\colon D^*(\tilde{Z})^\Gamma\to D^*(\tilde{Z})^\Gamma$.
		Then  $\phi_*^t(\varrho_1)$
		gives a homotopy between $\varrho_2$ and $\varrho_1$.
	\end{proof}
	
	\subsection{Perturbed signature operator on manifolds with cylindrical ends}
	
	In this section we are going to check that the construction we made for $\varrho$ and 
	$\varrho_\Gamma$ are well defined on the structure set $\mathcal{S}^{TOP}(N)$.

	For this purpose we will use the results presented in \cite{Wahl,PS,PS2}, where the authors
	have developed the theory in the smooth setting. Their methods are rather abstract and they also hold in the Lipschitz context.

	In order to develop the theory for manifolds with cylindrical ends, we are going to use the same  notations as \cite[2.19]{PS2}.
	
	The geometrical setting is the following: let $f\colon M\to N$ and $f'\colon M'\to N$ be two topological structures for $N$; 
	let $W$ be a cobordism between $\partial_0 W=M$ and $\partial_1 W=M'$ and let $W_\infty$ be the manifold
	with the infinite semi-cylinder $\partial W\times \RR_{\leq0}$ attached to the boundary; 
	let $V=~N\times[0,1]$ and let $V_\infty$ be the complete cylinders with base $\partial V=N$;
	there is a homotopy equivalence $F\colon W_\infty \to V_\infty$ which has the product form
	$F_\partial\times \mathrm{id}_{\RR_{\leq0}}$ on the cylindrical ends, where $F_{\partial_0}=~f\colon M\to N$
	and $F_{\partial_1}=f'\colon M'\to N$, both of them being homotopy equivalences.

	Thanks to the results presented in \cite{Hil2} we have a well defined Lipschitz signature complex on  the manifold
	$X=~W_\infty\cup-V_\infty$. Notice that $\partial_0 X=Z$ and $\partial_1 X=Z'$.
	
	First of all we need a generalization of  Theorem \ref{hs} for manifolds with cylindrical ends.
	This result is given by \cite[Proposition 8.1]{Wahl}, where  
	a perturbation of the signature operator is associated    to the homotopy equivalence $F$. Such a perturbation makes the operator invertible, as in the usual case.
	
	\begin{remark}
		As well as in the case presented in Theorem \ref{hs}, the generalization developed in \cite[Proposition 8.1]{Wahl} is still valid in the Lipschitz setting.
	\end{remark}
	
	The goal of this section is to check that the $\varrho$-class is well defined on the h-cobordism classes:
	as pointed out in \cite[Proposition 4.7]{PS2}, this is obtained by the
	combination of \cite[Theorem 8.4]{Wahl} and  \cite[Corollary 3.3]{PS2}.

	In \cite[Theorem 8.4]{Wahl}  all constructions work in the Lipschitz framework, where we
	do not consider the parameter $\varepsilon$. Wahl builds a perturbation  $\mathcal{C}_{F}^{\mathrm{cyl}}$ of the signature operator, supported on the cylindrical ends,
	from the perturbations on $Z$ and $Z'$;
	hence she constructs a homotopy of operators 
	between $\mathcal{D}_X+\mathcal{C}_{F}^{\mathrm{cyl}}$ and an other operator that, thanks to the Bunke's relative index theorem,
	has vanishing index.
	
	For the proof of the equality we just mentioned, the only not obvious point in the Lipschitz case is  the one concerning the relative index theorem
	proved in \cite{Bu}, since what remains of the proof uses abstract theory of unbounded operators and spectral flow methods.
	
	It is worth formulating Bunke's Theorem in the Lipschitz case and giving a sketch of its proof.
	
	\subsubsection{Bunke's relative index theorem for Lipschitz manifolds}

	The idea of the theorem is the following: let $X$ be a manifold, let $E\to X$ be a bundle  and $D$ a Fredholm operator on the sections of this
	bundle; if there exists a hypersurface $Y$ in $X$ such that the operator is invertible near $Y$, we can cut the manifold (and the bundle) along 
	$Y$ and we can paste a semicylinder to the boundary of both parts obtained after cutting, extending the bundle and the operator  along the semicylinder.
	Then we obtain an operator whose index equals the index of the original operator.
	
	More precisely we are considering the following data:
	the Lipschitz manifold $X$ we have defined in the previous subsection, the  Hilbert module $L^2(X,\Lambda_\CC(X)\otimes\mathcal{F})$
	of  $L^2$-forms on $X$ twisted by the Mishchenko bundle, that we are going to denote by $\mathcal{H}^0$;
	a regular operator $G$ that is the  twisted Lipschitz signature
	operator,  possibly perturbed by a bounded operator; we suppose that there is a Lipschitz function with compact support $f\geq0$ and $(G^2+f)^{-1}\in~\mathbb{B}(\mathcal{H}^0,\mathcal{H}^2)$
	(here $\mathcal{H}^2$ is the maximal domain of the square of the signature operator).
	
	\begin{definition}
		Let $\mathrm{Lip}_K(X)$ be the set of bounded Lipschitz functions $h$ such that for all $\varepsilon>0$ there exists a compact $C\subset X$,
		with $||dh_{|X\setminus C}||_{L^\infty}<\varepsilon$.
		Let us call $C_K(X)$ the closure of $\mathrm{Lip}_K(X)$ in the sup-norm.
	\end{definition}
	
	For the comfort of the reader, we recall the theorem stated in the Lipschitz setting.
	Let $E_i \to X_i$, $i =1,2$, be the two $C^*(\Gamma)$-C* bundles $\Lambda_{\CC}(X_i)\otimes \mathcal{F}_i$,  with operator $G_i$, associated to them as above. Let $W_i\cup_{Y_i}V_i$ be a partition of $X_i$ where $Y_i$  is a compact hypersurface.
	Assume that there is a commutative diagram of isomorphisms of all structures
	\[
	\xymatrix{\Psi\colon & E_{1|U(Y_1)}\ar[r]\ar[d]&  E_{2|U(Y_2)}\ar[d] \\
		\psi\colon & U(Y_1) \ar[r] &  U(Y_2)\\
		\psi_{|Y_1} &\colon Y_1\ar[r]\ar[u] & Y_2\ar[u]}
	\]
	where  $U(Y_i)$ is a tubular neighbourhood of $Y_i$, for $i = 1, 2$. We cut $X_i$ at
	$Y_i$, glue the pieces together interchanging the boundary components and obtain
	$X_3 := W_1\cup_Y V_2$ and $X_4 := W_2\cup_Y V_1$. Moreover, we glue the bundles using
	$\Psi$, which yields A-C* bundles $E_3\to X_3$ and $E_4\to X_4$ and we assume
	that $G_i$, $i = 3,4$ are again invertible at infinity. We define $[X_i]$ as the class $[\mathcal{H}^0_i,\frac{G_i}{G_i^2+f}]\in KK(C_K(X_i),C^*(\Gamma))$.
	The algebra $C_K(X)$ is unital. Hence, there is an embedding $i \colon \CC \to C_K(X)$
	and an induced map
	\[i^* \colon KK( C_K(X),C^*_r(\Gamma) )\to  KK(\CC),C^*_r(\Gamma) ).\]
	Set $\{X_i\} := i^*[X_i] \in KK(\CC,C^*_r(\Gamma))$ for $i = 1,\dots,4$.
	
	\begin{theorem}[\cite{Bu}]\label{bunke}
		\[\{X_l\} + \{X_2\} -\{X_3\} - \{X_4\} = 0.\]
	\end{theorem}
	Here are two facts:
	\begin{itemize}
		\item  thanks to \cite[Theorem 7.1]{Tel1} we have the following Rellich-type result: 
		the inclusion $\mathcal{H}^2\hookrightarrow\mathcal{H}^0$ is compact;
		\item for any $f$ Lipschitz function compactly supported on $X$, the multiplication operator
		$f\colon \mathcal{H}^2\to \mathcal{H}^0$ is compact. 
		And this also holds for the Clifford multiplication by $\mathrm{grad}(f)$, the gradient of $f$.
	\end{itemize}
	
	Let $R(\lambda)$ be the bounded operator $(G^2+f+\lambda)^{-1}\in \mathbb{B}(\mathcal{H}^0,\mathcal{H}^2)$, for $\lambda\geq0$; 
	because of the Rellich-type result, we know that $R(\lambda)$ defines a compact operator in $\mathbb{B}(\mathcal{H}^0)$ and that there is a positive constant 
	$C$ such that $||R(\lambda)||\leq (C+\lambda)^{-1}$.
	
	\begin{lemma}
		The integral \[F= \frac{G}{\pi}\int_0^\infty \lambda^{-\frac{1}{2}}R(\lambda)d\lambda\] is convergent
		and defines an operator in $\mathbb{B}(\mathcal{H}^0)$.
	\end{lemma}
	
	\begin{lemma}\label{comm}
		The operator $[D,R(\lambda)]$ extends to a bounded operator that coincides with \[-R(\lambda)\mathrm{grad}(f)R(\lambda).\] 
		Moreover such an operator is compact.
	\end{lemma}
	\begin{proof}
		See \cite[Lemma 1.7 and Lemma 1.8]{Bu}.
	\end{proof}
	
	\begin{lemma}
		For any $h\in C_K(X)$, $h(F^2-I)\in \mathbb{K}(\mathcal{H}^0)$.
	\end{lemma}
	\begin{proof}
		We have 
		\begin{equation*}
			\begin{split}
				&\left(\frac{G}{\pi}\int_0^\infty \lambda^{-\frac{1}{2}}R(\lambda)d\lambda \right)\left(\frac{G}{\pi}\int_0^\infty \lambda^{-\frac{1}{2}}R(\lambda)d\lambda\right)=\\
				&\frac{G^2}{\pi^2}\left(\int_0^\infty \lambda^{-\frac{1}{2}}R(\lambda)d\lambda\right)^2+\frac{G}{\pi}\left[\int_0^\infty \lambda^{-\frac{1}{2}}R(\lambda)d\lambda,\frac{G}{\pi}\right]\int_0^\infty \lambda^{-\frac{1}{2}}R(\lambda)d\lambda=\\
				&\frac{G^2}{\pi^2}\left(G^2+f\right)^{-1}-\frac{G}{\pi}\int_0^\infty \lambda^{-\frac{1}{2}}R(\lambda)\mathrm{grad}(f)R(\lambda)d\lambda \int_0^\infty \lambda^{-\frac{1}{2}}R(\lambda)d\lambda\sim\\
				&\frac{G^2}{\pi^2}\left(G^2+f\right)^{-1},
			\end{split}
		\end{equation*}
		where in the third step we have used Lemma \ref{comm}. Here $\sim$ means ``equal modulo compacts''.
		Hence 
		\[
		h(F^2-I)\sim h\frac{f}{G^2+f}
		\]
		that is compact, since the multiplication by $f$ is compact.
	\end{proof}
	
	\begin{lemma}
		For any $h\in C_K(X)$, $[F,h]\in \mathbb{K}(\mathcal{H}^0)$.
	\end{lemma}
	\begin{proof}
		Since we chose  $G$ as a perturbation of the signature operator $D$ and since  the perturbation becomes compact under bounded transform,
		we have that
		\begin{equation*}
			\begin{split}
				[F,h]\sim&\left[\frac{D}{\pi}\int_0^\infty \lambda^{-\frac{1}{2}}R(\lambda)d\lambda,h\right]=\\
				&\frac{D}{\pi}\left[\int_0^\infty \lambda^{-\frac{1}{2}}R(\lambda)d\lambda,h\right]+\left[\frac{D}{\pi},h\right]\int_0^\infty \lambda^{-\frac{1}{2}}R(\lambda)d\lambda=\\
				&\frac{D}{\pi}\int_0^\infty \lambda^{-\frac{1}{2}}\left[R(\lambda),h\right]d\lambda+\mathrm{grad}(h)\int_0^\infty \lambda^{-\frac{1}{2}}R(\lambda)d\lambda\sim\\
				&\frac{D}{\pi}\int_0^\infty \lambda^{-\frac{1}{2}}\left[R(\lambda),h\right]d\lambda.
			\end{split}
		\end{equation*}
		The term in the last line is compact as in the proof of \cite[Lemma 1.12]{Bu}.
	\end{proof}
	
	\begin{lemma}
		Let $f$ and $f_1$ be two positive and compactly supported Lipschitz functions such that 
		$(G^2+f)^{-1},(G^2+f_1)^{-1}\in\mathbb{B}(\mathcal{H}^0,\mathcal{H}^2)$.
		Then the two associated operators $F,F_1$  differ each other by a compact operator.
	\end{lemma}
	\begin{proof}
		See \cite[Lemma 1.10]{Bu}.
	\end{proof}

	The lemmas we presented yield to the following result.
	\begin{proposition}
		The pair $(\mathcal{H}^0,F)$ defines a Kasparov $(C_K(X),C^*_r(\Gamma))$-module and its class in $KK(C_K(X),C^*_r(\Gamma))$ does not depend 
		on the choice of $f$.
	\end{proposition}
	
	After checking this technical part, the proof of Theorem \ref{bunke}
	is completely abstract and it follows in the Lipschitz case as in the smooth one.

	\bigskip
	
	\bigskip
	
	\bigskip

	Now we treat another fundamental result proved by Piazza ans Schick: the delocalized Atiyah-Patodi-Singer index theorem.
	As noticed in \cite[Section 5.2]{PS2}, 
	the proof of the delocalized APS index theorem is based on  abstract functional analysis  for unbounded operators on Hilbert spaces.
	The reader can check that it almost completely works  in the Lipschitz context as well as in the smooth one and we will not give   all the details again.
	
	The only proof to be modified is \cite[Prop 5.33]{PS2}.  Assume the context and the notation understood, then we state the following Proposition.
	\begin{proposition}
		Given a Dirac type operator $D$, the operator $(1 + D^2)^{-1}\colon L^2 \to H^2$ is a norm limit of finite propagation operators $G_n \colon L^2 \to H^2$ with the property that $[\varphi,G_n]\colon L^2 \to H^2$  is compact for
		any compactly supported continuous function on $M$.
	\end{proposition} 
	\begin{proof}
		It is an easy computation showing that
		\[
		\frac{1}{1+x^2}=\int^{+\infty}_{-\infty}\frac{e^{-|t|}}{2}e^{-itx}dt.
		\]
		Let $f\colon \RR\to\RR$ be a $C^\infty$ function such that
		\begin{itemize}
			\item $0\leq f\leq 1$,
			\item $f=1$ on a neighbourhood of $0$,
			\item $f$ has compact support.
		\end{itemize}
		
		Define $G_n=\int^{+\infty}_{-\infty}f\left(\frac{t}{n}\right)\frac{e^{-|t|}}{2}e^{-itD}dt$. 
		
		Finite propagation: since $f\left(\frac{t}{n}\right)\frac{e^{-|t|}}{2}$ has compact support, $G_n$ has finite propagation speed.
		
		Pseudolocality:  thanks to the above formula,
		$(1 + D^2)^{-1}-G_n=\int^{+\infty}_{-\infty}(1-f\left(\frac{t}{n}\right))\frac{e^{-|t|}}{2}e^{-itD}dt$. 
		Notice that $(1-f\left(\frac{t}{n}\right))\frac{e^{-|t|}}{2}$ is $C^\infty$ and moreover it is a rapidly decreasing function on the spectrum of $D$. By \cite[Prop 5.31]{Roe}, $(1 + D^2)^{-1}-G_n$ is a bounded operator from $H^m$ to $H^k$ for any $m,k\in\NN$, hence $G_n$ is pseudolocal because so is $(1 + D^2)^{-1}$.
		Indeed, using Jacobi identities for commutators and the fact that $[\varphi,D]=c(d\varphi)$, $[\varphi,(1 + D^2)^{-1}]=(1+D^2)^{-1} c(d\phi) D (1+D^2)^{-1} + (1+D^2)^{-1} D c(d\phi) (1+D^2)^{-1})$ is compact,
		because the Clifford multiplication $c(d\varphi)$ is compact.
		
		In fact we need less: it is sufficient to show that 
		$(1 + D^2)^{-1}-G_n$ is a bounded operator from $L^2$ to $H^3$ and then, by Rellich Theorem, the commutator $[\varphi,(1 + D^2)^{-1}-G_n]$ turns out to be a compact operator
		from $L^2$ to $H^2$. To prove this, we only need that the third derivative of $(1-f\left(\frac{t}{n}\right))\frac{e^{-|t|}}{2}$ has a bounded supremum norm  (less than being rapidly decreasing).
		
		In fact, under these hypotheses and by the properties of the Fourier transform, we get that
		\begin{equation}
			||(1 + D^2)^{-1}-G_n||_{L^2\to H^3}=\left| \left| \left(\left(1-f\left(\frac{t}{n}\right)\right)\frac{e^{-|t|}}{2}\right)'''\right| \right| _\infty
		\end{equation}
		is bounded.
		Moreover $\left(\left(1-f\left(\frac{t}{n}\right)\right)\frac{e^{-|t|}}{2}\right)'''$ is equal to 
		\[
		-\frac{1}{n^3}f'''\left(\frac{t}{n}\right)e^{-|t|}+\frac{3}{n^2}f''\left(\frac{t}{n}\right)|t|e^{-|t|}-\frac{3}{n}f'\left(\frac{t}{n}\right)e^{-|t|}-\left(1-f\left(\frac{t}{n}\right)\right)|t|^3e^{-|t|}
		\] 
		that clearly goes to zero as $n$ goes to infinity.
		This also holds in the Lipschitz case.
	\end{proof}
	
	Now we can state the delocalized Atiyah-Patodi-Singer index theorem, that also holds in the Lipschitz context.
	
	\begin{theorem}[\cite{PS2}]\label{APSdeloc}
		If $\,i\colon C^*(\widetilde{X})^\Gamma\hookrightarrow D^*(\widetilde{X})^\Gamma$ is the inclusion and
		$j_*\colon D^*(\partial\widetilde{ X})^\Gamma\to D^*(\widetilde{X})^\Gamma$ is the map induced by the inclusion 
		$j\colon\partial\widetilde{ X}\hookrightarrow \widetilde{X}$,
		we have
		\[
		i_*(\mathrm{Ind}_\Gamma(D_X+C_{F}^{\mathrm{cyl}}))=j_*(\varrho(D_{\partial X}+C_{F_\partial}))\in K_0(D^*(\widetilde{X})^\Gamma).
		\]
	\end{theorem}
	Using the functoriality of the classifying map $u\circ F\cup u\colon \widetilde{X}\to E\Gamma$ and the map
	$\Phi:=\pi_1\circ(F\cup-\mathrm{id}_{V\times[0,1]})$ we obtain
	\[
	i_*\widetilde{\Phi}_*(\mathrm{Ind}_\Gamma(D_X+C_{F}^{\mathrm{cyl}}))=\varrho({F_\partial})\in K_0(D^*(\widetilde{V})^\Gamma)
	\]
	
	\[
	i_*u_*\widetilde{\Phi}_*(\mathrm{Ind}_\Gamma(D_X+C_{F}^{\mathrm{cyl}}))=\varrho_\Gamma(F_\partial)\in K_0(D^*_\Gamma).
	\]
	
	Observe that $\varrho_\Gamma$ is additive on disjoint unions as $\partial X= Z\cup -Z'$
	and in particular that 
	\[
	\varrho(F_\partial)= \varrho(f)- \varrho(f').
	\]
	Combining this with \cite[Theorem 8.4]{Wahl},
	we finally have that
	\[
	\varrho(f)=\varrho(f'),
	\]
	and similarly for $\varrho_\Gamma$,
	hence they are well defined on $\mathcal{S}^{TOP}(N)$.

	\section{Mapping surgery to analysis: the odd dimensional case}
	
	We refer the reader to \cite[Section 4]{PS2} for the definitions that we are going to recall:
	\begin{itemize}
		\item  $\mathcal{N}(N )$ is the set of normal maps. Its elements are degree 1 normal maps $[f\colon M\to N]$
		where M
		is an oriented manifold. Two such maps are equivalent if there is a
		normal cobordism between them. There is a map $\beta\colon \mathcal{N}(N)\to K_n(N)$ such that  $[f\colon M\to N]$ goes to the class $f_*[\mathcal{D}_M]-[\mathcal{D}_N]\in K_*(N)$, where $[\mathcal{D}_M]$ and $[\mathcal{D}_V]$ are the K-homology classes of the signature operators. The map $\beta$  already appears in the work of Higson and Roe where it is proved to be well defined.
		\item The map $\mathrm{Ind}_\Gamma \colon L_{n+1}(\ZZ\Gamma) \to K_{n+1}(C^*_r(\Gamma))$ has been defined by Wahl, following the results of Hilsum-Skandalis \cite{HilSk} and Piazza-Schick \cite{PS}.
		Recall that an element $x\in L_{n+1}(\ZZ\Gamma) $ is represented by a quadruple
		$(F\colon W\to X\times[0, 1], u\colon X \to B\Gamma)$ with $W$ a cobordism between two  orientable manifolds $\partial_1W$ and $\partial_2W$,
		$X$ an  orientable manifold, $F \colon (W, \partial W)\to (X \times[0, 1],\partial (X\times[0, 1]))$ a degree one normal map of pairs,
		$f_1 := F_{|\partial_1W}$ and $f_2 := F_{|\partial_2W}$ oriented homotopy equivalences and $u\colon X \to B\Gamma$ a classifying map.
		Let
		$f = f_1 \sqcup f_2$ denote the restriction of $F$ to $\partial W$. Consider $Z := W \sqcup X\times[0, 1]$ and $Z_\infty$ the manifold obtained attaching an infinite cylinder to the boundary $\partial Z$.
		Since $f$ is a homotopy equivalence, one can perturb the signature operator only on the cylindrical ends  and the obtained operator has a
		well defined index in $K_{n+1}(C^*_r(\Gamma))$, that is the image of 	$(F\colon W\to X\times[0, 1], u\colon X \to B\Gamma)$  through the map $\mathrm{Ind}_\Gamma$.
		
	\end{itemize}
	Now we can state the main theorem.
	
	\begin{theorem}
		Let $N$ be an $n$-dimensional closed oriented topological manifold with fundamental group $\Gamma$.
		Assume that $n\geq5$ is odd. Then there is a commutative diagram with exact rows
		\[
		\xymatrix{L_{n+1}(\ZZ\Gamma)\ar[r]\ar[d]^{\mathrm{Ind}_\Gamma} & \mathcal{S}^{TOP}(N)\ar[r]\ar[d]^\varrho & \mathcal{N}^{TOP}(N)\ar[r]\ar[d]^\beta & L_{n}(\ZZ\Gamma)\ar[d]^{\mathrm{Ind}_\Gamma} \\
			K_{n+1}(C^*_r(\Gamma))\ar[r] & K_{n+1}(D^*(\widetilde{N})^\Gamma)\ar[r] & K_n(N)\ar[r] & K_{n}(C^*_r(\Gamma))}
		\]
		and through the classifying map $u\colon N\to B\Gamma$ of the universal cover $\widetilde{N}$ of $N$, we have the analogous commutative diagram that
		involves the universal Higson-Roe exact sequence
		\[
		\xymatrix{L_{n+1}(\ZZ\Gamma)\ar[r]\ar[d]^{\mathrm{Ind}_\Gamma} & \mathcal{S}^{TOP}(N)\ar[r]\ar[d]^{\varrho_\Gamma} &
			\mathcal{N}^{TOP}(N)\ar[r]\ar[d]^{\beta_\Gamma} & L_{n}(\ZZ\Gamma)\ar[d]^{\mathrm{Ind}_\Gamma} \\
			K_{n+1}(C^*_r(\Gamma))\ar[r] & K_{n+1}(D^*_\Gamma)\ar[r] & K_n(B\Gamma)\ar[r] & K_{n}(C^*_r(\Gamma))}
		\]
	\end{theorem}
	
	\begin{remark}
		Let us recall the fact that, despite  $\mathcal{S}^{TOP}(N)$ has a group structure, we do not deal with it and the top row is considered just as a sequence of sets, as in the smooth case.
	\end{remark}
	
	Thanks to the results of the previous section we can check that the results in  \cite[Sections 4.2 and 4.3]{PS2} hold in terms of the category of topological manifolds
	instead of the one of smooth manifolds: all proofs are still valid in the Lipschitz context.
	Thanks to the work of C.Whal \cite[Theorem 9.1]{Wahl}, that can be combined with Theorem \ref{Hil}, the first vertical arrow is well defined in the Lipschitz setting. The second one is also well defined for the previous section. Concerning the third one there are no significant problems.
	
	The same method used in Proposition\ref{dep} applies to the 
	class of the signature and its index class, then all vertical arrows do not depend on the chosen Lipschitz structure.
	
	One has to check the commutativity of the three squares.
	\begin{itemize}
		\item The third square is obviously commutative: let  $(f\colon M\to N)$ be a normal map in $\mathcal{N}^{TOP}(N)$, it is sent horizontally to the same map forgetting that it is normal and then
		through $\mathrm{Ind}_\Gamma$ to the difference $\mathrm{Ind}_\Gamma(D_M)-\mathrm{Ind}_\Gamma(D_N)$;
		on the other hand $\beta(f\colon M\to N)=f_*[D_M]-[D_N]$,
		that gives, through the analytic assembly map,  the
		index class just founded.
		\item Let us study the second square: let $(f\colon M\to N)$ be a structure in $\mathcal{S}^{TOP}(N)$, it goes to
		the same map forgetting that $f$ is a homotopy equivalence;
		the $\varrho$-class $\varrho(f)$, as in Definition \ref{rho}, 
		is the push-forward through $\tilde{\varphi}$ of the class
		\[\left[\frac{1}{2}(1+\chi(\widetilde{D}_Z+C_f))\right]\in K_0(D^*(\widetilde{Z})^\Gamma);\]
		this goes horizontally to the class in $ K_0(D^*(\widetilde{Z})^\Gamma/C^*(\widetilde{Z})^\Gamma)$
		that represents, by Paschke duality, the K-homology class
		of the signature operator of $Z$; then by functoriality of $\tilde{\varphi}_*$ and the fact that $\beta(f\colon M\to N)=f_*[D_M]-[D_N]$, we obtain the commutativity of the second square.
		
		\item For the commutativity of first square, the proof is exactly the same as in  \cite[4.10]{PS2}.
		Let $a\in L_{n+1}(\ZZ\Gamma)$ and let $(f\colon M\to N)$ be a structure in $\mathcal{S}^{TOP}(N)$.
		The commutativity of the first square means that the following  equation holds:
		\[
		i_*(\mathrm{Ind}_\Gamma(a))=\varrho(a[f\colon M\to N])-\varrho([f\colon M\to N])\in K_0(D^*(\widetilde{Z})^\Gamma);
		\]
		this is proved identifying the right hand side with the class predicted by the APS delocalized index theorem, that, as we know, holds in the Lipschitz case too.
		The proof is based on an addition formula, as in \cite[7.1]{Wahl}, and algebraic identifications of $\varrho$-classes, that the reader can check still holding, word-for-word, in the Lipschitz case.
	\end{itemize}

	\section{Products}
	
	Let $M$ and $N$ be two Cartesian products with a common factor, namely $M=M_1\times M_2$ and $N=N_1\times M_2$, and let
	$f_1\colon M_1\to N_1$ be a homotopy equivalence. Therefore $f=f_1\times\mathrm{id}\colon M\to N$ is a homotopy equivalence.
	
	Observe that the signature operator on $Z=M\cup(-N)$ has this form: $D_Z=D_1\hat{\otimes}1+1\hat{\otimes}D_2$,
	i.e. the graded tensor product of the signature operator $D_1$ on $M_1\cup(-N_1)$ and the signature operator $D_2$ on $M_2$.
	
	As before we construct from $f$ a bounded operator $C_f$ that  produces an invertible perturbation $D_Z+C_f$.
	Notice that, from the construction in \cite{HilSk} and as it has been pointed out in \cite[(6.1)]{Wahl}, the operator $C_f$ has the form $C_{f_1}\hat{\otimes}1$, where all grading operators are understood in the graded tensor product. We have 
	\[
	D_Z+C_f=(D_1+C_{f_1})\hat{\otimes}1+1\hat{\otimes}D_2
	\]
	so we can associate  an invertible perturbation of $D_Z$ to an invertible perturbation of $D_1$.
	
	We would like to state a product formula involving the $\varrho$-class invariant
	of the first factor and the K-homology class of the second one. For this aim it will be useful to give another realization of the group $K_*(D^*(\widetilde{X})^\Gamma)$.

	In \cite[II.2]{Ka}  the Grothendieck group of a functor $\varphi\colon\mathcal{C}\to \mathcal{C}'$ is defined as the 
	set of triples $(E,F,\alpha)$, where $E$ and $F$ are objects in the category $\mathcal{C}$ and $\alpha$ is an isomorphism
	$\varphi(E)\to\varphi(F)$ in the category $\mathcal{C}'$, modulo the following equivalence relation:
	two triples $(E,F,\alpha)$ and $(E',F',\alpha')$ are equivalent if there exist two isomorphisms $f\colon E\to E'$ and $g\colon F\to F'$ such that
	the following diagram 
	\[
	\xymatrix{\varphi(E)\ar[r]^{\alpha}\ar[d]^{\varphi(f)} & \varphi(F)\ar[d]^{\varphi(g)}\\
		\varphi(E')\ar[r]^{\alpha'} & \varphi(F')}
	\]
	commutes.
	
	In \cite[II.3.28]{Ka} it is shown that, when $\varphi$ is the restriction of vector bundles over a space $X$ to a closed subspace $Y$,
	one obtains the relative K-group $K(X,Y)$ as the K-theory of the mapping cone of the inclusion $i\colon Y\hookrightarrow X$.
	
	In \cite{SkExt}, G. Skandalis used the same idea: considering  an element $x$ in $KK(A,B)$ as a functor from $K(A)$ to $K(B)$ through the Kasparov product, one can construct a relative K-group  $K(x)$ and one can also prove that it is isomorphic to the K-theory of a mapping cone C*-algebra.
	Moreover this relative K-group fits in a long exact sequence
	\[
	\xymatrix{\dots\ar[r] & K(B\otimes C_0(0,1))\ar[r] & K(x)\ar[r] & K(A)\ar[r] &\dots }
	\] 
	such that the boundary map is given by the Kasparov product with $x$.

	More generally, if we fix a separable C*-algebra $D$, we can consider an element $x$ in $KK(A,B)$ as a functor from $KK(D,A)$ to $KK(D,B)$, through the Kasparov product with $x$ on the right and we can 
	still obtain a relative KK-group $K(D,x)$ that turns out to be isomorphic to the group $KK(D,C_\psi)$, where $C_\psi$ is the mapping cone C*-algebra of a suitable *-homomorphism $\psi$.
	
	So we have seen that constructing relative KK-groups corresponds, in a philosophical way, to taking the Grothendieck group of a functor or, in a more concrete way, to taking the Grothendieck group of a mapping cone. We want to construct a long exact sequence of groups
	such that the boundary map is the  assembly map. Notice that the difficulty resides in the fact that the assembly map is not induced by a morphism nor by a Kasparov product on the right. But it is  still possible to construct a group. 
	
	\subsection{The analytic structure set and products}
	
	Let $X$ be a proper and cocompact $\Gamma$-space, we would like to give an explicit construction of the cycles of $K_*(D^*(X)^\Gamma)$ in terms of Kasparov bimodules, so that 
	one can define a product by means of Kasparov products.
	
	In \cite{RoeC} J. Roe shows that the following diagram 
	\begin{equation}	\label{roecomparing}
		\xymatrix{K_{*+1}(D^*(X)^\Gamma/C^*(X)^\Gamma)\ar[r]^-{\mathrm{Ind}}\ar[d]^{P} & K_{*}(C^*(X)^\Gamma)\ar[d]^{\cong}\\
			KK^*_\Gamma(C_0(X),\CC)\ar[r]^{\mu^{\Gamma}_{X}}& K_*(C^*_r(\Gamma))}
	\end{equation}
	is commutative. Here $P$ is given by the Paschke duality and $\mu_X^\Gamma$ is the assembly map defined by Kasparov in \cite{kasp}.
	Let us recall some notation and definitions.
	For any $\Gamma$-C*-algebras $A$ and $B$, there exists a
	descent homomorphism 
	\[j^\Gamma\colon KK_\Gamma(A,B)\to KK(A\rtimes \Gamma,B\rtimes \Gamma)\]
	which is functorial and compatible with respect to Kasparov products.
	It associates to an equivariant KK-cycle $[H,\phi,F]$ the Kasparov bimodule $[H\rtimes \Gamma, \widetilde{\phi},\widetilde{F}]$, where
	
	\begin{itemize}
		\item $H\rtimes \Gamma$ is the $A\rtimes \Gamma$-$B\rtimes \Gamma$-bimodule given by the completion of $C_c(\Gamma,H)$, with the usual $C_c(\Gamma,B)$-valued inner product and left $C_c(\Gamma,A)$-action;
		\item $\widetilde{\phi}$ is the extension to $A\rtimes\Gamma$ of the representation of
		$C_c(\Gamma,A)$ induced by $\phi$ ;
		\item $\widetilde{F}$ is the extension to  $H\rtimes \Gamma$ of the operator $F$  that  associates to
		$\gamma\mapsto\alpha(\gamma)$ the element $\gamma\mapsto F(\alpha(\gamma))$ on $C_c(\Gamma,H)$.
	\end{itemize}
	
	Moreover we know that for any proper and cocompact $\Gamma$-space $X$
	one can construct an imprimitivity $C(X/\Gamma)$-$C_0(X)\rtimes\Gamma$-bimodules $E_X$. Since $C(X/\Gamma)$ is unital $E_X$ defines an element in $KK(\CC,C_0(X)\rtimes\Gamma)$.
	
	\begin{definition}\label{defassembly}
		The assembly map $\mu^\Gamma_X$ is defined as the composition
		
		\[
		\xymatrix{KK^*_\Gamma(C_0(X),\CC)\ar[r]^-{j^\Gamma}& KK^*(C_0(X)\rtimes\Gamma,C^*_r(\Gamma))\ar[r]^-{[E_X]\otimes-} & K_*(C^*_r(\Gamma))}.
		\]
	\end{definition}
	
	Let us sketch  how the commutativity of the diagram \eqref{roecomparing} is proved: if $F\in\BB(L^2(X))$ is an element of $D^*(X)^\Gamma$ such that 
	$F^2=1$ and $F^*=F$ modulo $C^*(X)^\Gamma$, a representative for \[\mu_X^\Gamma\left[L^2(X),\varphi\colon C_0(X)\to\BB(L^2(X)),F\right]\] is given in the following way.
	Since we can assume that $F$ is exactly of finite propagation, it defines an operator $F_c$ on the pre-Hilbert space $L^2_c(X)$ of the compactly supported $L^2$-functions of $X$.
	One can endow $L^2_c(X)$ with the following $\CC\Gamma$-valued inner product
	\[
	\langle f,g\rangle_{\CC\Gamma}(\gamma)=\langle f^\gamma,g\rangle_\CC
	\]
	where $\langle f^\gamma,g\rangle_\CC$ is the standard inner product between $g$ and the function $f$ translated by $\gamma$.
	With a standard double completion of the pair $(L^2_c(X),\CC\Gamma)$ we obtain an Hilbert module over $C^*_r(\Gamma)$ that we denote by $L^2_\Gamma(X)$.
	Now $F_c$ extends to an adjointable operator $\widetilde{F}$ on $L^2_\Gamma(X)$ and the class
	$[L^2_\Gamma(X),1\otimes\widetilde{\varphi},\widetilde{F}]\in KK(\CC,C^*_r(\Gamma))$ is equal to $\mu_X^\Gamma\left[L^2(X),\varphi,F\right]$.
	\begin{remark}\label{invertibleF}
		Notice that if $F$ is invertible, then $\widetilde{F}$ is also invertible. 
	\end{remark}
	
	Moreover one can prove that 
	$L^2_\Gamma(X)$ is a complemented sub-Hilbert module of $L^2(X)\rtimes\Gamma$. In fact if $\phi\colon X\to [0,1]$ is a compactly supported function such that
	\[
	\sum_{\gamma\in\Gamma}(\phi^2)^\gamma=1,
	\]
	then the projection 
	\[
	p=\sum_{\gamma\in\Gamma}\phi\cdot\phi^{\gamma^{-1}}[\gamma]\in C_0(X)\rtimes\Gamma
	\] 
	has as range the $C^*_r(\Gamma)$-module $L^2_\Gamma(X)$. Actually  the projection $p$  gives the class  $[E_X]\in KK(\CC,C_0(X)\rtimes\Gamma)$ used in  the Definition \ref{defassembly} of the assembly map.

	\begin{definition}\label{defcycle}
		Let $X$ be as above. A $\Gamma$-equivariant analytic structure cycle on $X$ consists of the following data:
		\begin{itemize}
			\item an equivariant selfadjoint Kasparov bimodule $(H,\phi,T)\in \mathbb{E}^{\Gamma}(C_0(X),\CC)$ ;
			\item a  Kasparov bimodule $(\mathcal{E}(t),\psi(t),S(t))\in \mathbb(\CC,C^*_r(\Gamma)[0,1))$, such that $\mathcal{E}(0)=E_{X}\otimes_{C_0(X)\rtimes\Gamma}H \rtimes \Gamma$, $\psi(0)=\mathrm{id}\otimes_{\tiny{C_0(X)\rtimes\Gamma}}\phi$ (that from now on we will denote in short by $\mathrm{id}\otimes\phi$), $S(0)$ is a $\widetilde{T}$-connection  and $S(1)$ is invertible. Here $\widetilde{\phi}$ and $\widetilde{T}$ are as in the definition of the descent homomorphism.
			That is the class of $(\mathcal{E}(0),\psi(0),S(0))$ is equal to $\mu_X^\Gamma(H,\phi,T)$.
			
		\end{itemize}
		
		Such a cycle is said to be degenerate if both $(H,\phi,T)$ and $(\mathcal{E}(t),\psi(t),S(t))$ are degenerate Kasparov bimodules.
	\end{definition}
	
	\begin{definition}
		Let $(H_i,\phi_i,T_i,\mathcal{E}(t)_i,\psi(t)_i,S_i(t))$, $i=0,1$, 
		be two $\Gamma$-equivariant analytic structure cycles. 
		
		We will say that they are homotopic if
		there exists a path $(H_s,\phi_s,T_s,\mathcal{E}_s(t),\psi(t)_s,S_s(t))$ of  $\Gamma$-equivariant analytic structure cycles that joins them. 
		Then we denote by $S_j^\Gamma(X)$ the Grothendieck group generated by all homotopy
		classes of $\Gamma$-equivariant analytic structure cycles on $X$.
		
		We can define in an analogous way a group $S^{\Gamma}_*(X,A)$, where $A$ is any $\Gamma$-C*-algebra, using the assembly map with coefficient $\mu_{X,A}^{\Gamma}\colon KK^\Gamma(C_0(X),A)\to KK(\CC,A\rtimes\Gamma)$. 
	\end{definition}

	\begin{remark}
		Note that one can give the definition of $S^\Gamma_*(X,A)$ also when $\Gamma$ is a groupoid instead of a group.
	\end{remark}

	\begin{proposition}\label{commutativo}
		There is a commutative diagram 
		\[
		\xymatrix{
			\cdots\ar[r] & 
			K_j(C^*(X)^\Gamma)\ar[r]\ar[d]^{\beta}  & K_j(D^*(X)^\Gamma)\ar[r]\ar[d]^{\alpha} & K_j(D^*(X)^\Gamma/C^*(X)^\Gamma)\ar[r]\ar[d]^{P}& \cdots\\
			\cdots\ar[r] & 
			KK^{j-1}(\CC,C_r^*(\Gamma)\otimes C_0(0,1))\ar[r]& S_{j-1}^\Gamma(X)\ar[r]& KK_\Gamma^{j-1}(C_0(X),\CC)\ar[r]& \cdots
		}
		\]
		whose vertical arrows are isomorphisms.
	\end{proposition} 
	\begin{proof}
		
		Let $\left(L^2(X),\phi\colon C_0(X)\to \BB(L^2(X))\right)$ be the $\Gamma$-equivariant $C_0(X)$-module used to construct the algebra $D^*(X)^\Gamma$.
		The map $P$ is given by the Paschke duality.
		The homomorphism $\beta$ is given by the composition of the isomorphism between $C^*(X)^\Gamma$ and $C^*_r(\Gamma)$, and the Bott  periodicity.
		Let us describe the homomorphism $\alpha\colon K_0(D^*(X)^\Gamma)\to K_{1}(\mu^\Gamma_X)$.
		It associates to a projection $p$ over  $D^*(X)^\Gamma$ the cycle
		$(H,\varphi,F,\mathcal{E}(t),\psi(t),S(t))$, where
		\begin{itemize}
			\item $(H,\varphi,F)=(L^2(X),\varphi,2p-1)$;
			\item $(\mathcal{E}(t),\psi(t),S(t))$ is given by the path constantly equal to $(L^2_\Gamma(X),\mathrm{id}\otimes\widetilde{\varphi}, \widetilde{F})$, that is the triple  built in  the discussion at the beginning of the present section.		
		\end{itemize}
		
		Observe that $L^2_\Gamma(X)$ is nothing but $E_X\otimes_{C_0(X)\rtimes\Gamma}L^2(X)\rtimes\Gamma$ and that $\widetilde{F}$ is invertible by construction (see Remark \ref{invertibleF}).
		
		The homomorphism $\beta$ associates to a projection $p$ over $C^*(X)^\Gamma$ the Kasparov bimodule
		$[L^2_\Gamma(X),\mathrm{id}\otimes\widetilde{\phi},G(t)]$, where  $G(t)$ is  the loop of invertible elements $\widetilde{F}(1-e^{2i\pi t})-1$  over $C^*(X)^\Gamma$, given by the Bott periodicity.

		The second square is obviously commutative.
		Concerning the first one, since $(L^2(X),\phi, F)$ is degenerate as Kasparov bimodule, it is easy to produce a homotopy of cycles
		between $(0,0,0,L^2_\Gamma(X),\mathrm{id}\otimes\widetilde{\varphi},G(t))$ and 
		$\alpha(i_*[p])=[H,\varphi,F,L^2_\Gamma(X),\mathrm{id}\otimes\widetilde{\varphi},S(t)]$, where $S(t)$ is the constant path equal to $\widetilde{F}$.
		To  do that, observe that $G\left(\frac{1}{2}\right)=\widetilde{F}$ and that 
		$G_s(t)=G\left((1-s)t+\frac{1}{2}s\right)$ does the job.	
	\end{proof}
	
	\begin{remark}\label{intuition}
		Let $\varphi\colon A\to B$ a C*-algebras morphism and $C_\varphi$ its mapping cone.
		Then we obtain naturally the long exact sequence of K-groups
		\[
		\xymatrix{
			\cdots\ar[r] & 
			K_*(SB)\ar[r]  & K_*(C_\varphi)\ar[r] & K_*(A)\ar[r]^{\varphi_*}& K_*(B)\ar[r]& \cdots}
		\]
		whose boundary morphism is induced by $\varphi$.
		Conversely, if we start from a homomorphism $K_*(A)\to K_*(B)$ induced by a morphism $\varphi\colon A\to B$, then this homomorphism fits into a sequence as above.

		As explained before, the idea behind the construction of $S^\Gamma_j(X)$ is considering the assembly map as a functor. But instead of a Kasparov product on the right, we have the assembly map, that is the composition of the descent morphism and a Kasparov product on the left, and instead of the K-theory of a mapping cone of C*-algebras we obtain the Grothendieck group of  
		a "mapping cone" of Kasparov bimodules.
	\end{remark}

	\begin{definition}\label{Sfunc}
		Let $Y,Z$ be two spaces and assume that a group $\Gamma$ acts on $Y$ and $Z$ in a proper and cocompact way.
		Let $f\colon Y\to X$ be a $\Gamma$-equivariant continuous map.
		We can define  a homomorphism
		\[f_*\colon S_{j}^\Gamma(Y)\to S_{j}^\Gamma(X)\]
		such that $f_*[H,\phi,T,\mathcal{E}(t),\psi(t),S(t)]=[H,\phi\circ f,T,\mathcal{E}'(t),\psi'(t),S'(t))]$.
		
		Here $(\mathcal{E}'(t),\psi'(t),S'(t))$ is the concatenation of the path we are going to describe and the path $(\mathcal{E}(t),\psi(t),S(t))$.
		The first one is the path connecting the Kasparov bimodules $(E_{X}\otimes_{C_0(X)\rtimes\Gamma}H\rtimes \Gamma,\mathrm{id}\otimes(\widetilde{\phi\circ f)}, S')$ and $(E_{Y}\otimes_{C_0(Y)\rtimes\Gamma}H\rtimes \Gamma,\mathrm{id}\otimes\widetilde{\phi}, S)$, 
		where $S'$ is a $\widetilde{T}$-connection on $E_{Y}\otimes_{C_0(Y)\rtimes\Gamma}H\rtimes \Gamma$. This path always exists  thanks to the functoriality of the assembly map: since $\mu^\Gamma_Y=\mu^\Gamma_X\circ f_*\colon KK_\Gamma(C_0(Y))\to K_0(C^*_r(\Gamma))$, it turns out that $(E_{X}\otimes_{C_0(X)\rtimes\Gamma}H\rtimes \Gamma,\mathrm{id}\otimes(\widetilde{\phi\circ f}), S')$ and $(E_{Y}\otimes_{C_0(Y)\rtimes\Gamma}H\rtimes \Gamma,\mathrm{id}\otimes\widetilde{\phi}, S)$ define  the same class.
		
		Moreover the class obtained does not depend on the choice of this path.

	\end{definition}
	\begin{lemma}\label{coarsemorphism}
		Let $Y,X$ be two Riemannian manifolds and assume that a group $\Gamma$ acts on $Y$ and $X$ freely, isometrically and such that $Y/\Gamma$ is compact.
		Let $f\colon Y\to X$ be a $\Gamma$-equivariant continuous coarse map and let $V\colon H_Y\to H_X$ be an isometry that covers $f$ in the $D^*$-sense (\cite[Definition 1.7]{PS}).
		Then the following diagram
		
		\[
		\xymatrix{
			K_j(D^*(Y)^\Gamma)\ar[r]^{f_*}\ar[d]^{\alpha} & K_j((D^*(X)^\Gamma)\ar[d]^{\alpha}\\
			S_{j-1}^\Gamma(Y)\ar[r]^{f_*} & S_{j-1}^\Gamma(X)
		}
		\]
		commutes.
	\end{lemma}
	\begin{proof}
		
		Let $(H_Y,\phi_Y\colon C_0(Y)\colon \BB(H_Y))$ and $(H_X,\phi_X\colon C_0(X)\colon \BB(H_X))$ the representations used to construct the algebras $D^*(Y)^\Gamma$ and $D^*(X)^\Gamma$. 
		Let $p$ be a projection over $ D^*(Y)^\Gamma$.
		Remember that $f_*[p]=[\mathrm{Ad}_V(p)]\in K_0(D^*(X)^\Gamma)$.
		Then we get two elements of $S_{1}^{\Gamma}(X)$:
		\begin{itemize}
			\item
			the first one is $\alpha([\mathrm{Ad}_V(p)])=[H_X,\phi_X,T,\mathcal{E}_X,\mathrm{id}\otimes\widetilde{\phi}_X,S]$. Here $T=2\mathrm{Ad}_V(p)-1$, $\mathcal{E}_X=E_{X}\otimes_{C_0(X)\rtimes\Gamma}H_X\rtimes \Gamma$ and $S$ is the path constantly equal to a $\widetilde{T}$-connection;
			\item the second one is $f_*(\alpha[p])=[H_Y,\phi_Y\circ f^*,U,\mathcal{E}'(t),\psi'(t),S'(t)]$. Here $U=2p-1$ 
			and $(\mathcal{E}'(t),\psi'(t),S'(t))$ is the path connecting $(E_{X}\otimes_{C_0(X)\rtimes\Gamma}H\rtimes \Gamma,\mathrm{id}\otimes(\widetilde{\phi\circ f)}, S')$ and $(E_{Y}\otimes_{C_0(Y)\rtimes\Gamma}H\rtimes \Gamma,\mathrm{id}\otimes\widetilde{\phi}, S)$, 
			where $S'$ is a $\widetilde{U}$-connection on $E_{Y}\otimes_{C_0(Y)\rtimes\Gamma}H\rtimes \Gamma$.
		\end{itemize}
		We have to prove that these two classes are the same.
		
		Consider the projection $Q=VV^*$, then we can decompose $\alpha([\mathrm{Ad}_V(p)])$ in two direct summands:
		\[\alpha([\mathrm{Ad}_V(p)])=[QH_X,\phi_X,T_1,R\mathcal{E}_X,\mathrm{id}\otimes\widetilde{\phi}_X,S_1]\oplus[(1-Q)H_X,\phi_X,T_2,(1-R)\mathcal{E}_X,\mathrm{id}\otimes\widetilde{\phi}_X,S_2],\]
		where $T_1=QTQ$, $T_2=(1-Q)T(1-Q)$, $R$ is a $\widetilde{Q}$-connection and $S_1$ and $S_2$ are defined similarly.
		
		The second summand is clearly degenerate and, since to the following diagram
		\[
		\xymatrix{C_0(X)\ar[r]^{f^*}\ar[drr]_{\varphi_X} & C_0(Y)\ar[r]^{\varphi_Y} & \mathbb{B}(H_Y)\ar[d]^{\mathrm{Ad}_V}_{\simeq}\\
			& & \mathbb{B}(QH_X)}
		\]
		commutes, modulo compacts operators, the first one
		is equal to $f_*(\alpha[p])$.
	\end{proof}

	\begin{remark}\label{Khom}
		If we define the group $\hat{K}^\Gamma_j(X)$ as in definition \ref{defcycle}, but dropping the condition of $S(1)$ being invertible, we get that the map
		\[\hat{K}^\Gamma_j(X)\ni [H,\phi,T,\mathcal{E}(t),\psi(t),S(t)]\mapsto [H,\phi,T]\in KK_\Gamma^j(C_0(X),\CC) \]
		is a group isomorphism.
		Indeed one can easily check that the kernel of this map is isomorphic to $KK_j(\CC,C^*_r(\Gamma)\otimes C_0(0,1])$, that is trivial since $C^*_r(\Gamma)\otimes C_0(0,1]$ is a cone.
		The inverse map is obviously given by 
		\[[H,\phi,T]\mapsto[H,\phi,T,\mathcal{E},\mathrm{id}\otimes\widetilde{\phi},S] \]
		where $S$ is the path constantly equal to any $\widetilde{T}$-connection.
	\end{remark}

	\begin{definition}
		Let $\Gamma$ be a discrete group, we can define the following exact sequence of groups
		\[
		\xymatrix{\dots\ar[r] & KK_*(\CC,C^*_r(\Gamma)\otimes C_0(0,1) )\ar[r]& S^\Gamma_*\ar[r]& \hat{K}^\Gamma_*\ar[r]& \dots}
		\]
		as the direct limit of
		\[\xymatrix{\dots\ar[r] & KK_*(\CC,C^*_r(\Gamma)\otimes C_0(0,1) )\ar[r]& S^\Gamma_*(X)\ar[r]& \hat{K}^\Gamma_*(X)\ar[r]& \dots}\]
		over all cocompact $\Gamma$-subspaces $X$ of $E\Gamma$.
	\end{definition}
	
	Thus we obtain the same groups defined in \cite[Definition 1.11]{PS}. This follows easily from Proposition \ref{commutativo} and Lemma \ref{coarsemorphism}.

	Let \[\xi=[H_1,\phi_1,T_1,\mathcal{E}_1(t),\psi_1(t),S_1(t)]\in S^{\Gamma_1}_j(X_1)\]
	and let \[\lambda=[H_2,\phi_2,T_2,\mathcal{E}_2(t),\psi_2(t),S_2(t)]\in \hat{K}^{\Gamma_2}_i(X_2),\] where $X_1$ and $X_2$ are two proper and cocompact spaces with respect to $\Gamma_1$ and $\Gamma_2$ respectively.
	Let $(H_1\hat{\otimes}H_2,\phi_1\hat{\otimes}\phi_2,T)$ be an exterior Kasparov product of $(H_1,\phi_1,T_1)$ and $(H_2,\phi_2,T_2)$.
	Let $(\mathcal{E}(t),\psi(t),S(t))$ be the restriction to the diagonal of the  Kasparov product of $(\mathcal{E}_1(t),\psi_1(t),S_1(t))$ and $(\mathcal{E}_2(t),\psi_2(t),S_2(t))$ (that is a Kasparov $\CC$-$A$-bimodule, where $A$ is equal to the algebra $C^*_r(\Gamma_1)\otimes C^*_r(\Gamma_2)\otimes C_0([0,1]^2\setminus\{1\}\times[0,1])$).


	\begin{definition}\label{product}
		We define a product
		\[
		S^{\Gamma_1}_j(X_1)\times \hat{K}^{\Gamma_2}_i(X_2)\to S^{\Gamma_1\times\Gamma_2}_{j+i}(X_1\times X_2)
		\]
		that associates to $\xi\times\lambda$ the class  \[\xi\boxtimes\lambda:=[H_1\hat{\otimes}H_2,\phi_1\hat{\otimes}\phi_2,T, \mathcal{E}(t),\psi(t),S(t)]\]
		where the entries are as described above.
		The product is compatible with homotopies in both factors and so it is well defined.
	\end{definition}
	
	\begin{remark}\label{diagrammaI}
		A similar product is defined in an obvious way on $KK^{j-1}(\CC,C^*_r(\Gamma_1)\otimes C_0(0,1))$ and $\hat{K}^{\Gamma_1}_{j}(C(X_1),\CC)$. It is natural in the sense that the
		following diagram
		\[
		\xymatrix{\cdots\ar[r] & KK^{j}(\CC,A)\times \hat{K}^{\Gamma_2}_i(X_2)\ar[r]\ar[d]& S_{j}^{\Gamma_1}(X_1)\times \hat{K}^{\Gamma_2}_i(X_2)\ar[r]\ar[d]& \hat{K}^\Gamma_{j}(X_1)\times \hat{K}^{\Gamma_2}_i(X_2)\ar[r]\ar[d]& \cdots\\
			\cdots\ar[r] & KK^{j+i}(\CC,B)\ar[r]& S_{j+i}^{\Gamma_1\times\Gamma_2}(X_1\times\widetilde{X}_2)\ar[r]& \hat{K}^{\Gamma_1\times\Gamma_2}_{j+i}(X_1\times X_2))\ar[r]& \cdots}
		\]
		is commutative.
		Here $A=C^*(\widetilde{X}_1)^{\Gamma_1}\otimes C_0(0,1)$ and $B=C^*(\widetilde{X}_1\times\widetilde{X}_2)^{\Gamma_1\times\Gamma_2}\otimes C_0(0,1)$.
	\end{remark} 
	
	\begin{lemma}\label{functprod}
		Let $Y,X,Z$ be three spaces and assume that a group $\Gamma_1$ acts properly and cocompactly on $Y$ and $X$ and $\Gamma_2$ acts properly and cocompactly on $Z$. Let $f\colon Y\to X$ be a $\Gamma$-equivariant continuous map.
		Then the following diagram
		\[
		\xymatrix{
			S_i^{\Gamma_1}(Y)\times \hat{K}_j^{\Gamma_2}(Z)\ar[r]^{f_*\times id}\ar[d]& S_i^{\Gamma_1}(X)\times \hat{K}_j^{\Gamma_2}(Z)\ar[d] \\
			S_{j+i}^{\Gamma_1\times\Gamma_2}(Y\times Z)\ar[r]^{(f\times id_Z)_*} & S_{j+i}^{\Gamma_1\times\Gamma_2}(X\times Z)
		}
		\]
		where the vertical arrows are given by \ref{product}, is commutative.
	\end{lemma}
	\begin{proof}
		This is straightforward since 
		$(\phi_1\otimes\phi_2)\circ (f\times id_Z)^*= (\phi_1\circ f^*)\otimes \phi_2$.
	\end{proof}
	\subsection{Stability of $\varrho$ classes}
	
	\subsubsection{The signature operator}
	Let $f\colon M\to N$ be a structure in $\mathcal{S}^{TOP}(N)$ and $\varrho(f)$ be the associated $\varrho$-class in $K_*(D^*(\widetilde{Z})^\Gamma)$.
	Let us see the different realisations of this class with respect to the different models of the analytical structure set.
	
	\begin{itemize}
		\item In $K_0(D^*(\widetilde{Z})^\Gamma)$  we have the element $\left[\frac{1}{2}(1+\chi(\widetilde{D}_Z+C_f))\right]$.         
		\item In $S^\Gamma_1(\widetilde{Z})$ this  element turns into  \[\left[H,\phi,F,\mathcal{E},\mathrm{id}\otimes\widetilde{\phi}, G\right],\] where $F=\left(\chi(\widetilde{D}_Z+C_f)\right)$, $\mathcal{E}=E_{Z}\otimes_{C_0(Z)\rtimes\Gamma}H\rtimes \Gamma$ and $G$ is the path constantly equal to the $\widetilde{F}$-connection used in the proof of Proposition \ref{commutativo}.
		\item Finally observe that the image of the last element 
		through the natural map $S^\Gamma_1(\widetilde{Z})\to S^\Gamma_1$ is the image of $\varrho_\Gamma\in K_0(D^*_\Gamma)$ by means of the obvious isomorphism. 
		
	\end{itemize}

	\begin{proposition}\label{prodrho}
		Let $M_1$ and $N_1$ be two $n$-dimensional Lipschitz manifolds with $n$ odd and let $M_2$ be an $m$-dimensional Lipschitz manifold with $m$ even.
		Let $M$ be $M_1\times M_2$, let $N$ be $N_1\times M_2$ and let
		$f_1\colon N_1\to M_1$ be a homotopy equivalence. Let $\Gamma_i$ be the fundamental groups of $M_i$, with $i=1,2$.
		We have that
		\[
		\varrho(f_1\times\mathrm{id}_{M_2})=\varrho(f_1)\boxtimes[D_2]\in S^{\Gamma_1\times\Gamma_2}_1(\widetilde{M}_1\times\widetilde{M}_2)
		\]
		and the same holds for $\varrho_\Gamma$.
	\end{proposition}
	
	\begin{proof}
		Let $Z_1=M_1\cup N_1$ and $Z_2= M_1\times M_2\cup N_1\times M_2$. 
		
		The class $\varrho(f_1)$ is represented in $S^{\Gamma_1}_1(Z_1)$ by the cycle \[\left[H_1,\phi_1,F_1,\mathcal{E}_1,\mathrm{id}\otimes\widetilde{\phi}_1, G_1\right],\] where $F_1=\chi(\widetilde{D}_{Z_1}+C_{f_1})$.
		
		The class $[D_2]\in \hat{K}_1^{\Gamma_2}(M_2)$ is represented by \[\left[H_2,\phi_2,F_2,\mathcal{E}_2,\mathrm{id}\otimes\widetilde{\phi}_2, G_2\right],\] where $F_2=\psi(\widetilde{D}_{M_2})$.
		
		Finally the class $\varrho(f_1\times \mathrm{id}_{M_2})\in S^{\Gamma_1\times\Gamma_2}_1(Z_1\times M_2)$ is represented by
		\[\left[H_1\otimes H_2,\phi_1\otimes\phi_2,F,\mathcal{E}_1\otimes\mathcal{E}_2,\mathrm{id}\otimes\widetilde{\phi}_1\otimes\widetilde{\phi}_2, G\right],\] where $F=\chi(\widetilde{D}_{Z_1}\otimes1+1\otimes\widetilde{D}_{M_1}+C_{f_1\times\mathrm{id}_{M_2}})$.
		
		We have to prove the identity of the last class mentioned with the product $\varrho(f)\boxtimes[D_2]\in S^{\Gamma_1\times\Gamma_2}_1(Z_1\times M_2)$ given by
		\[
		\left[H_1\otimes H_2,\phi_1\otimes\phi_2,F',\mathcal{E}_1\otimes\mathcal{E}_2,1\otimes\widetilde{\phi}_1\otimes\widetilde{\phi}_2, G'\right],
		\]
		where $F'=\chi(\widetilde{D}_{Z_1}+C_f)\otimes1+1\otimes\psi(\widetilde{D}_{M_2})$.
		
		Since $\widetilde{D}_{Z_1}\otimes1+1\otimes\widetilde{D}_{M_1}+C_{f_1\times\mathrm{id}_{M_2}}=(\widetilde{D}_{Z_1}+C_f)\otimes1+1\otimes\widetilde{D}_{M_2}$ 
		and that $\chi$ and $\psi$  differ by a function in $C_0(\RR)$, the identity follows from 
		\cite{BJ}.

		Trivially this holds for $\varrho_\Gamma$ too.
	\end{proof}

	We would like that, after fixing a non zero K-homology class $\lambda$, under suitable assumptions the product with this element is an injective map. 
	
	To prove that, we need to define a new group that
	we will denote by $\mathcal{T}_*^{\Gamma_1,\Gamma_2}(X_1,X_2)$ (notice that the order of $X_1$ and $X_2$ is not irrelevant).
	\begin{definition}\label{triangle}
		A cycle of $\mathcal{T}_j^{\Gamma_1,\Gamma_2}(X_1,X_2)$ consists of the following data:
		\begin{itemize}
			\item a  Kasparov bimodule $(H,\phi,T)\in\mathbb{E}^{\Gamma_1\times\Gamma_2}\left(C_0(X_1)\otimes C_0(X_2),\CC\right)$;
			\item   a  Kasparov bimodule $(\mathcal{E}_s,\psi_s, S_s)\in\mathbb{E}^{\Gamma_1}\left(C_0(X_1),C^*_r(\Gamma_2)\otimes C[0,1]\right)$, where $\mathcal{E}_0$ is equal to $E_{X_2}\otimes_{C_0(X_2)\rtimes\Gamma_2}H\rtimes\Gamma_2$, $\psi_0=\mathrm{id}\otimes\widetilde{\phi}$ and $S_0$ is any $\widetilde{T}$-connection;
			\item a Kasparov bimodule $(\mathcal{E}'_{t,s},\psi'_{t,s}, S'_{t,s})\in\mathbb{E}\left(\CC,C^*_r(\Gamma_1)\otimes C^*_r(\Gamma_2)\otimes C_0(\mathcal{T})\right)$, where
			$\mathcal{T}$ is the triangle $\left\{(t,s)\in[0,1]^2\setminus\{1,1\}\,|\, t\leq s \right\}$, $\mathcal{E}'_{0,s}=E_{X_1}\otimes_{C_0(X_1)\rtimes\Gamma_1}\mathcal{E}_s\rtimes\Gamma_1$, $\psi'_{0,s}=\mathrm{id}\otimes\widetilde{\psi}$ and $S'_{0,s}$ is any $\widetilde{S}_s$-connection;
			
		\end{itemize}
		modulo homotopies of cycles, defined in a obvious way.
	\end{definition}

	\begin{remark}
		To have an intuition of what this group is, accordingly with the idea in Remark \ref{intuition}, one can think of it as the restriction to the triangle $\mathcal{T}=\left\{(t,s)\in[0,1]^2\setminus\{1,1\}\,|\, t\leq s \right\}$ of the product of the "mapping cone" $\mu^{\Gamma_1}_{X_1}$ and the "mapping cylinder" of $\mu^{\Gamma_2}_{X_2}$. This idea was used in Definition \ref{product} too.

	\end{remark}
	
	\begin{lemma}
		The group $\mathcal{T}_*^{\Gamma_1,\Gamma_2}(X_1,X_2)$ is isomorphic to
		$S^{\Gamma_1\times\Gamma_2}(X_1\times X_2)$.
	\end{lemma}
	
	\begin{proof}
		Define the homomorphism $\Phi\colon \mathcal{T}_*^{\Gamma_1,\Gamma_2}(X_1,X_2)\to S^{\Gamma_1\times\Gamma_2}(X_1\times X_2)$ 
		given by
		\[
		\left((H,\phi,T),(\mathcal{E}_s,\psi_s, S_s),(\mathcal{E}'_{t,s},\psi'_{t,s}, S'_{t,s})\right)\mapsto
		\left(H,\phi,T,\mathcal{E}'_{t,t},\psi'_{t,t}, S'_{t,t}\right).
		\]
		Define the following homomorphism
		\[\Psi\colon	\left(H,\phi,T,\mathcal{H}_{t},\alpha_{t}, U_{t}\right)\mapsto	\left((H,\phi,T),(\mathcal{E}_s,\psi_s, S_s),(\mathcal{E}'_{t,s},\psi'_{t,s}, S'_{t,s})\right) \]
		where 
		\begin{itemize}
			\item$(\mathcal{E}_s,\psi_s, S_s)$ is the path constantly equal to $(E_{X_2}\otimes_{C_0(X_2)\rtimes\Gamma_2}H\rtimes\Gamma_2, \mathrm{id}\otimes\widetilde{\phi}, S)$, with $S$ any $\widetilde{T}$-connection; 
			\item for all fixed $t\in[0,1)$,  $(\mathcal{E}'_{t,s},\psi'_{t,s}, S'_{t,s})$ is the paths constantly equal to
			$(\mathcal{H}_{t},\alpha_{t}, U_{t})$.
		\end{itemize}
		It is easy to check that the third condition in Definition \ref{triangle} is satisfied and that $\Phi$ and $\Psi$ are inverse to each other.
		
	\end{proof}	
	\begin{proposition}\label{injectivity}
		Let $\lambda$ be a class in $ \hat{K}^{\Gamma_2}_i(X_2)$.
		If there exists a class $\zeta\in KK^{-i}(C^*_r(\Gamma_2),\CC)$ such that $\mu^{\Gamma_2}_{X_2}(\lambda)\otimes_{C^*_r(\Gamma_2)}\zeta=n$ with $n\neq0$, then \[\boxtimes \lambda \colon S_i^{\Gamma_1}(X_1)\otimes\ZZ\left[\frac{1}{n}\right] 
		\to S_{i+j}^{\Gamma_1\times \Gamma_2}(X_1\times X_2)\otimes\ZZ\left[\frac{1}{n}\right] \] is injective.
		In particular if $\mu^{\Gamma_2}_{X_2}(\lambda)\otimes_{C^*_r(\Gamma_2)}\zeta=1$, then the product with $\lambda$ is honestly injective.
	\end{proposition}
	\begin{proof}
		To prove the Lemma we are going to build a left inverse for $\boxtimes\lambda$.
		Define the map $c_\zeta$ as the composition of the following ones:
		\begin{itemize}
			\item the isomorphism $\Psi\colon S_*^{\Gamma_1\times\Gamma_2}(X_1\times X_2)\to  \mathcal{T}_*^{\Gamma_1,\Gamma_2}(X_1,X_2)$,
			\item the evaluation at $s=1$, 
			$\mathrm{ev_{s=1}}\colon \mathcal{T}_*^{\Gamma_1,\Gamma_2}(X_1,X_2)\to S^{\Gamma_1}_*(X_1,C^*_r(\Gamma_2)) $,
			given by
			\[
			\left((H,\phi,T),(\mathcal{E}_s,\psi_s, S_s),(\mathcal{E}'_{t,s},\psi'_{t,s}, S'_{t,s})\right)\mapsto
			\left((\mathcal{E}_1,\psi_1, S_1),(\mathcal{E}'_{t,1},\psi'_{t,1}, S'_{t,1})\right)
			\]
			\item the morphism $S^{\Gamma_1}_*(X_1,C^*_r(\Gamma_2))\to S^{\Gamma_1}_{*-i}(X_1) $ given by
			\[
			(H,\phi,T,\mathcal{E}(t),\psi(t),S(t))\mapsto(H',\phi',T',\mathcal{E}'(t),\psi'(t),S'(t)),
			\]
			where  $(H',\phi',T')$ is any Kasparov product of $(H,\phi,T,\mathcal{E}(t))$ and $\zeta$, and $(\mathcal{E}'(t),\psi'(t),S'(t))$ is any Kasparov product of $(\mathcal{E}(t),\psi(t),S(t))$ and $\zeta$.
			
		\end{itemize}
		
		It is easy to check that $\mathrm{ev}_{s=1}\circ\Psi\circ\boxtimes\lambda\colon S_{i}^{\Gamma_1}(X_1)\to S_{i+j}^{\Gamma_1}(X_1,C^*_r(\Gamma_2))$ is just the exterior product with $\mu_{X_2}^{\Gamma_2}(\lambda)$.
		Then, by hypothesis, $c_\zeta(x\boxtimes\lambda)= n\cdot x$ for any $x\in S^{\Gamma_1}_i(X_1)$.
		After inverting $n$, we get an inverse for $\boxtimes\lambda$.  
	\end{proof}
	
	\begin{remark}
		The same argument fits to prove that if we fix an element $x\in \hat{K}_i^{\Gamma_2}(X_2)$ satisfying the above condition, then 
		the vertical arrows of the following diagram
		\[\tiny{
			\xymatrix{\cdots\ar[r] & KK^{j}(\CC,A)\ar[r]\ar[d]^{\boxtimes x}& S_{j}^{\Gamma_1}(X_1)\ar[r]\ar[d]^{\boxtimes x}& \hat{K}^{\Gamma_1}_1(X_1)\ar[r]\ar[d]^{\boxtimes x}& \cdots\\
				\cdots\ar[r] & KK^{j+i}(\CC, B)\otimes\ZZ\left[\frac{1}{n}\right] \ar[r]& S_{j+i}^{\Gamma_1\times\Gamma_2}(X_1\times X_2)\otimes\ZZ\left[\frac{1}{n}\right] \ar[r]& \hat{K}^{\Gamma_1\times\Gamma_2}_1(X_1\times X_2)\otimes\ZZ\left[\frac{1}{n}\right] \ar[r]& \cdots}}
		\]
		are rationally injective. Here $A=C^*(X_1)^{\Gamma_1}\otimes C_0(0,1)$ and $B=C^*(X_1\times X_2)^{\Gamma_1\times\Gamma_2}\otimes C_0(0,1)$.
	\end{remark}
	
	We can obtain the condition of Lemma \ref{injectivity} under certain hypotheses on  $\Gamma_2$:
	we impose that the group has a $\gamma$ element,
	this means that there exists a C*-algebra on which $\Gamma$ acts properly and elements
	\[
	\eta\in KK_{\Gamma}(\CC,A)\quad\mbox{and}\quad d\in KK_{\Gamma}(A,\CC),
	\]
	such that $\gamma=\eta\otimes_A d\in KK_{\Gamma}(\CC,\CC)$ satisfies $p^*\gamma=1\in KK_{\underline{E}\Gamma\rtimes\Gamma}(C_0(\underline{E}\Gamma),C_0(\underline{E}\Gamma))$, where $\underline{E}\Gamma$ is the classifying space for proper actions of $\Gamma$ and $p\colon \underline{E}\Gamma\rtimes\Gamma\to\Gamma$ is the homomorphism defined by $p(z,g)=g$.  We refer the reader to \cite{JLT1,JLT2}.
	
	The existence of the $\gamma$ element implies that the Baum-Connes assembly map (with coefficients) is split injective and that the group is K-amenable: this last property gives the existence of a non trivial element
	$\zeta\in KK(C^*_r(\Gamma_2),\CC)$
	such that, if $\xi=[L^2(\widetilde{X}_2),D]\in KK_{\Gamma_2}(\widetilde{X}_2,\CC)$ is the class given by an equivariant elliptic operator $D$, then $\mu_{\widetilde{X}_2}^{\Gamma_2}(D)\otimes_{C^*_r(\Gamma_2)}\zeta$ is equal to the Fredholm index of the induced operator on $\widetilde{X}_2/\Gamma$.
	
	\begin{corollary}\label{stab}
		Let $M_2$ be an even dimensional Lipschitz manifold with fundamental group $\Gamma_2$ 
		such that it has a $\gamma$ element and $[D_2]\in K_*(M_2)$ has non zero index.
		If $f_1\colon N_1\to M_1$ and $f'_1\colon N'_1\mapsto M_1$ are homotopy equivalences between odd dimensional Lipschitz manifolds,
		with different  $\varrho$-class invariants, then 
		\[
		[f_1\times\mathrm{id}_{M_2}]\neq[f'_1\times\mathrm{id}_{M_2}]\in \mathcal{S}^{TOP}(M_1\times M_2).
		\]
	\end{corollary}
	
	\subsubsection{Dirac operators and positive scalar curvature}\label{subsection}
	We would like to apply the methods of the previous sections to 
	get similar results about the secondary invariants described in \cite{PS}.
	
	Let us recall \cite[Definition 1.6]{PS}:
	let $(M,g)$  be a Riemannian spin manifold of dimension $n>0$,
	with fundamental group $\Gamma$.
	If $g$ has uniformly positive scalar curvature then the Dirac operator $\slashed{D}_M$ is invertible and $\chi(\widetilde{\slashed{D}}_M)$, the bounded transform of the lift of $\slashed{D}_M$ to the universal covering of $M$, defines a class $\varrho_g\in D^*(\widetilde{M})^{\Gamma} $.
	
	Thanks to that and the APS-delocalized Theorem, for $n$ odd, one obtains the following commutative diagram
	\[
	\xymatrix{\Omega^{\mathrm{spin}}_{n+1}(M)\ar[r]\ar[d]^\beta & \mathrm{R}^{\mathrm{spin}}_{n+1}(M)\ar[r]\ar[d]^{\mathrm{Ind}_\Gamma} & \mathrm{Pos}^{\mathrm{spin}}_{n}(M)\ar[r]\ar[d]^\varrho & \Omega^{\mathrm{spin}}_{n}(M)\ar[d]^\beta \\
		K_{n+1}(M)\ar[r]& K_{n+1}(C^*_r(\Gamma))\ar[r] & K_{n+1}(D^*(\widetilde{M})^\Gamma)\ar[r] & K_n(M) }
	\]
	where $M$ is a compact space with fundamental group $\Gamma$ and universal covering $\widetilde{M}$. The first row in the diagram is the Stolz exact sequence, see for instance \cite[Definition 1.39]{PS}.
	
	In the $S^\Gamma_*(M)$ picture of the analytic structure set, the class $\varrho_g$ is given by the quadruple \[[L^2(M,\slashed{S}), C(M), \slashed{D}_M, \widetilde{\chi(\slashed{D}_M)}].\] Here the last term is the constant path $\widetilde{\chi(\slashed{D}_M)}$ because the operator is invertible and there is no need to perturb it. 
	
	\begin{remark}
		If $(M,g)$ has positive scalar curvature  and $(N,h)$ is another Riemannian manifold, then for $\varepsilon>0$ small enough, $(M\times N, g\times \varepsilon h)$ has positive scalar curvature.
		Hence if $M$ admits a metric with positive scalar curvature, so does $M\times N$.
	\end{remark}
	\begin{proposition}\label{proddirac}
		Let $M$ be a spin manifold of dimension $n$  and let $g$ be a Riemannian metric with positive scalar curvature on $M$. 
		Let $N$ be a spin manifold of dimension $m$  and $h$ a Riemannian metric such that $(M\times N,g\times h)$ has positive scalar curvature.
		Then 
		\[ \varrho_g\boxtimes [\slashed{D}_h]=\varrho_{g\times h}\in S_{n+m}^{\Gamma_1\times\Gamma_2}(M\times N),\]
		where $\Gamma_1$ and $\Gamma_2$ are the fundamental groups of $M$ and $N$ respectively and $[\slashed{D}_h]$ is the class of the Dirac operator on $N$ in $K_m(N)$.
	\end{proposition}
	\begin{proof}
		We can prove the result as in \ref{prodrho}.
		Moreover since the class $\varrho_g$ is represented by a quadruple whose last term is the constant path $\widetilde{\chi(\slashed{D}_h)}$, it turns out that we can prove it in an easier way (see for instance \cite[Proposition 6.2.13]{Sieg}).
	\end{proof}

	\begin{corollary}\label{stabdirac}
		Let $M$ be a spin manifold of odd dimension $n$  with fundamental group $\Gamma_1$ and let  $g_1$ and $g_2$ be two Riemannian metrics with positive scalar curvature on $M$ such that $\varrho_{g_1}\neq \varrho_{g_2}\in S_n^\Gamma(M)$. 
		Let $(N,h)$ be a Riemannian spin manifold of even dimension $m$  with fundamental group $\Gamma_2$, such that the index of $[\slashed{D}_N]\in K_m(N)$ is $k\neq0$, $\Gamma_2$ has a $\gamma$ element and $g_i\times h$ has positive scalar curvature on $M\times N$, for $i=1,2$.

		Then \[[g_1\times h]\neq[g_2\times h]\in \mathrm{Pos}^{\mathrm{spin}}_{n+m}(M\times N)\otimes\ZZ\left[\frac{1}{k}\right] .\]
	\end{corollary}
	\begin{proof}
		We can use the arguments we used for Lemma
		\ref{injectivity} to  obtain immediately the result. 
	\end{proof}

	\subsection{The delocalized APS index Theorem in the odd-dimensional case}

	Another application of the product formula is the proof of the delocalized APS index theorem for odd dimensional cobordisms.
	
	We will do it for the perturbed signature operator, the theorem for the Dirac operator on a spin manifold with positive scalar curvature is completely analogous.
	
	Because of motivations well explained in \cite[Remark 4.6]{PS2}, we will prove the theorem at the cost of inverting $2$.
	We recall that here and in \cite{PS2} the signature operator on an odd dimensional manifold is not the odd signature operator of Atiyah, Patodi and Singer, but the direct sum of two (unitarily equivalent) versions of this operator.
	
	Since in the statement of the delocalized APS index theorem in the odd dimensional case we will compare the $\varrho$ invariant of the boundary with the index of the APS odd signature operator on the cobordism, it is worth to specify the notation we shall follow: on an odd dimensional manifold we denote by $D^{\mathrm{APS}}$ the odd signature operator of Atiyah, Patodi and Singer and we denote by $D$ the odd signature operator that we used so far.
	
	The strategy of the proof is to reduce the odd dimensional case to the even dimensional one through the product by the K-homology class of the signature operator on the circle.
	Then it is useful to review the behavior of the signature operator with respect to cartesian products of manifolds. For a detailed treatment we refer the reader to sections $5$ and $6$ of \cite{Wprod}.

	Let $W$ be an $n$-dimensional manifold with boundary $\partial W$ endowed with a cocompact free $\Gamma$-action. We assume that $n$ is odd and that the boundary of $W$ is composed by a pair of homotopy equivalent manifolds.  Let $j\colon \partial W\hookrightarrow W$ and $j'\colon \partial W\times\RR\hookrightarrow W\times\RR$ be the obvious inclusions.
	Let us recall some useful facts:
	\begin{itemize}
		\item the even signature operator $D_{W\times S^1}$ is equivalent to the direct sum of two copies of the exterior product
		$D^{APS}_{W}\hat{\otimes}1+1\hat{\otimes}D^{APS}_{S^1}$, see \cite[Section 6.3]{Wprod}. Since $D_{S^1}$ is the sum of two equivalent versions of $D_{S^1}^{APS}$, one has that $D_{W\times S^1}$ is equivalent to 
		$D^{APS}_{W}\hat{\otimes}1+1\hat{\otimes}D_{S^1}$.
		Consequently the higher index of $(D_{W\times S^1}+C_{F\times\mathrm{id}}^{\mathrm{cyl}})^+$ is equal to the
		class given by the product $\frac{1}{2} [\mathrm{Ind}_\Gamma(D^{APS}_W+C_{F}^{\mathrm{cyl}})]\boxtimes [D_{S^1}]$, where here $\boxtimes\colon K_{i}(C^*_r(\Gamma))\times K_j(S^1)\to K_{i+j}(C^*_r(\Gamma\times\ZZ))$;
		\item the operator $D^{APS}_{\partial W\times S^1}$ is equivalent to the exterior product of the even dimensional signature operator $D_{\partial W}$ and the odd dimensional signature operator $D_{S^1}^{APS}$, see \cite[Section 6.1]{Wprod}. Thus we obtain that the odd dimensional operator 
		$D_{\partial W\times S^1}$ is equivalent to  the exterior product of the even dimensional signature operator $D_{\partial W}$ and the odd dimensional signature operator $D_{S^1}$.
		In particular this means that $\varrho(D_{\partial W}+ C_{F_{\partial}})\boxtimes [D_{S^1}]$ is equal to $\varrho(D_{\partial W\times S^1}+C_{F_\partial\times\mathrm{id}})$, where here $\boxtimes\colon S^\Gamma_i(\widetilde{W})\times K_j(S^1)\to S^{\Gamma\times\ZZ}_{i+j}(\widetilde{W}\times\RR)$.
	\end{itemize}
	
	\begin{remark}
		Notice that, since $D^{APS}_{S^1}$ is nothing else than the Dirac operator on the circle and since $D_{S^1}$ is unitarily equivalent to two copies of  $D^{APS}_{S^1}$, its index is two times the generator of $C^*_r(\ZZ)$. Now
		$KK(C^*_r(\ZZ),\CC)\cong KK(C(S^1),\CC)$ by Fourier transform and
		$KK(C(S^1),\CC)\cong \mathrm{Hom}(K_0(C(S^1)),\ZZ)$, by \cite[Theorem 7.5.5]{HRk} for instance.
		So choosing any homomorphism from $K_0(C(S^1))$ to $\ZZ$ that sends the index of $D^{APS}_{S^1}$ to $1$, we obtain a class 
		$\zeta\in KK(C^*_r(\ZZ),\CC)$ that satisfies the assumptions of Lemma \ref{injectivity}, with $n=2$.
		
	\end{remark}

	\begin{theorem}
		If $\,i\colon C^*(\widetilde{W})^\Gamma\hookrightarrow D^*(\widetilde{W})^\Gamma$ is the inclusion and
		$j_*\colon D^*(\partial\widetilde{ W})^\Gamma\to D^*(\widetilde{W})^\Gamma$ is the map induced by the inclusion 
		$j\colon\partial\widetilde{ W}\hookrightarrow \widetilde{W}$,
		we have
		\[
		i_*\left(\frac{1}{2}\mathrm{Ind}_\Gamma(D^{\mathrm{APS}}_W+C_{F}^{\mathrm{cyl}})\right)=j_*(\varrho(D_{\partial W}+C_{F_\partial}))\in K_0(D^*(\widetilde{W})^\Gamma)\otimes\ZZ\left[\frac{1}{2}\right],
		\]
		where $\frac{1}{2}\mathrm{Ind}_\Gamma(D^{APS}_W+C_{F}^{\mathrm{cyl}})\in K_0(C^*(\widetilde{W})^\Gamma)\otimes\ZZ\left[\frac{1}{2}\right]$.
	\end{theorem}
	
	\begin{proof}
		Let $W$ be as above. Because of Proposition \ref{commutativo} and Lemma \ref{coarsemorphism} we will prove the theorem in the $S^\Gamma_{*}(\cdot)$ setting.
		
		Let $\Pi_D\colon S^\Gamma_{0}(\widetilde{W})\otimes\ZZ\left[\frac{1}{2}\right]\to S^{\Gamma\times \ZZ}_{0}(\widetilde{W}\times\RR)\otimes\ZZ\left[\frac{1}{2}\right] $ and $\Pi_C\colon K_1(C^*(\widetilde{W})^\Gamma)\otimes\ZZ\left[\frac{1}{2}\right]\to K_1(C^*( \widetilde{W}\times\RR)^{\ZZ\times\Gamma})\otimes\ZZ\left[\frac{1}{2}\right] $ be the morphism induced by the product with the of class of the signature operator $D_{S^1}$ in $K_1(S^1)$.
		By Lemma \ref{injectivity}, we have that 
		\begin{equation}\label{apsodd}
			i_*\left(\frac{1}{2}\mathrm{Ind}_\Gamma(D^{APS}_W+C_{F}^{\mathrm{cyl}})\right)=j_S(\varrho(D_{\partial W}+C_{F_\partial}))\end{equation}
		holds if and only if 
		\[ \Pi_D\left(i_*\left(\frac{1}{2}\mathrm{Ind}_\Gamma(D^{APS}_W+C_{F}^{\mathrm{cyl}})\right)\right)=\Pi_D(j_S(\varrho(D_{\partial W}+C_{F_\partial})))
		\]
		holds.

		But by Remark \ref{diagrammaI} it turns out that 
		\[
		\Pi_D\left(i_*\left(\frac{1}{2}\mathrm{Ind}_\Gamma(D^{APS}_W+C_{F}^{\mathrm{cyl}})\right)\right)=
		i_*\left(\Pi_C\left(\frac{1}{2}\mathrm{Ind}_\Gamma(D_W^{APS}+C_{F}^{\mathrm{cyl}})\right)\right)
		\]
		and, and by \cite[Section 6.3]{Wprod}, that 
		\[
		\Pi_C\left(\frac{1}{2}\mathrm{Ind}_\Gamma(D^{APS}_W+C_{F}^{\mathrm{cyl}})\right)=\mathrm{Ind}_\Gamma(D_{W\times S^1}+C_{F\times\mathrm{id}}^{\mathrm{cyl}}).
		\]
		
		Moreover by Lemma \ref{functprod} it follows that 
		\[
		\Pi_D(j_S(\varrho(D_{\partial W}+C_{F_\partial})))
		= j'_S(\Pi_D(\varrho(D_{\partial W}+C_{F_\partial})))
		\]
		and, by Proposition \ref{proddirac}, that 
		\[
		\Pi_D\left(\varrho(D_{\partial W}+C_{F_\partial})\right)=
		\varrho(D_{\partial W\times S^1}+C_{F_\partial\times\mathrm{id}}).
		\]

		Thus we have that
		\eqref{apsodd}
		holds if and only if
		\[ i_*\left(\mathrm{Ind}_\Gamma(D_{W\times S^1}^{APS}+C_{F\times\mathrm{id}}^{\mathrm{cyl}})\right)=j'_S\left(\varrho(D_{\partial W\times S^1}+C_{F_\partial\times\mathrm{id}})\right)
		\]
		holds. But, since $W\times S^1$ is even dimensional, the equality  on the right-hand side holds by \ref{APSdeloc} and the Theorem is proved.
	\end{proof}
	
	If $W$ is a Spin Riemannian manifold with boundary, such that the metric on the boundary has positive scalar curvature, then we can state the analogous theorem for the $\varrho$ invariants associated to Dirac operators.
	
	\begin{theorem}\label{apspsc}
		If $\,i\colon C^*(\widetilde{W})^\Gamma\hookrightarrow D^*(\widetilde{W})^\Gamma$ is the inclusion and
		$j_*\colon D^*(\partial\widetilde{ W})^\Gamma\to D^*(\widetilde{W})^\Gamma$ is the map induced by the inclusion 
		$j\colon\partial\widetilde{ W}\hookrightarrow \widetilde{W}$,
		we have
		\[
		i_*(\mathrm{Ind}_\Gamma(\slashed {D}_W))=j_*(\varrho(\slashed {D}_{\partial W}))\in K_0(D^*(\widetilde{W})^\Gamma).
		\]
	\end{theorem}
	
	Notice that in this case it is not necessary to invert $2$. Moreover the proof of the theorem is very similar to the case of the signature operator, but easier because we do not have to perturb the Dirac operator to obtain an invertible operator.

	\section{Mapping surgery to analysis: the even dimensional case}
	The extension to the odd dimensional case of the delocalized APS index theorem allows us to state the following result (with proof almost identical 
	to the the one given in the odd dimensional case).
	
	\begin{theorem}
		Let $N$ be an $n$-dimensional closed oriented topological manifold with fundamental group $\Gamma$.
		Assume that $n\geq5$ is even. Then there is a commutative diagram with exact rows
		\[
		\xymatrix{L_{n+1}(\ZZ\Gamma)\ar[r]\ar[d]^{\mathrm{Ind}_\Gamma} & \mathcal{S}^{TOP}(N)\ar[r]\ar[d]^\varrho & \mathcal{N}^{TOP}(N)\ar[r]\ar[d]^\beta & L_{n}(\ZZ\Gamma)\ar[d]^{\mathrm{Ind}_\Gamma} \\
			K_{n+1}(C^*_r(\Gamma))\otimes\ZZ\left[\frac{1}{2}\right]\ar[r] & K_{n+1}(D^*(\widetilde{N})^\Gamma)\otimes\ZZ\left[\frac{1}{2}\right]\ar[r] & K_n(N)\otimes\ZZ\left[\frac{1}{2}\right]\ar[r] & K_{n}(C^*_r(\Gamma))\otimes\ZZ\left[\frac{1}{2}\right]}
		\]
		and through the classifying map $u\colon N\to B\Gamma$ of the universal cover $\widetilde{N}$ of $N$, we have the analogous commutative diagram that
		involves the universal Higson-Roe exact sequence
		\[
		\xymatrix{L_{n+1}(\ZZ\Gamma)\ar[r]\ar[d]^{\mathrm{Ind}_\Gamma} & \mathcal{S}^{TOP}(N)\ar[r]\ar[d]^{\varrho_\Gamma} &
			\mathcal{N}^{TOP}(N)\ar[r]\ar[d]^{\beta_\Gamma} & L_{n}(\ZZ\Gamma)\ar[d]^{\mathrm{Ind}_\Gamma} \\
			K_{n+1}(C^*_r(\Gamma))\otimes\ZZ\left[\frac{1}{2}\right]\ar[r] & K_{n+1}(D^*_\Gamma)\otimes\ZZ\left[\frac{1}{2}\right]\ar[r] & K_n(B\Gamma)\otimes\ZZ\left[\frac{1}{2}\right]\ar[r] & K_{n}(C^*_r(\Gamma))\otimes\ZZ\left[\frac{1}{2}\right]}
		\]
	\end{theorem}

	The same is true if we consider the surgery exact sequence for smooth manifolds.

	\begin{remark}
		Thanks to Theorem \ref{apspsc}, we can enunciate the analogous statement for the Stolz sequence.
		With the same notations as in subsection \ref{subsection}, we obtain the following commutative diagram
		\[
		\xymatrix{\Omega^{\mathrm{spin}}_{n+1}(M)\ar[r]\ar[d]^\beta & \mathrm{R}^{\mathrm{spin}}_{n+1}(M)\ar[r]\ar[d]^{\mathrm{Ind}_\Gamma} & \mathrm{Pos}^{\mathrm{spin}}_{n}(M)\ar[r]\ar[d]^\varrho & \Omega^{\mathrm{spin}}_{n}(M)\ar[d]^\beta \\
			K_{n+1}(M)\ar[r]& K_{n+1}(C^*_r(\Gamma))\ar[r] & K_{n+1}(D^*(\widetilde{M})^\Gamma)\ar[r] & K_n(M) }
		\]
		with $n\geq5$ even.
	\end{remark}
	
	\begin{corollary}
		Corollaries \ref{stab} and \ref{stabdirac} are true irrespective of the dimesions of $M$ and $N$.
	\end{corollary}
	\addcontentsline{toc}{section}{References}
	\bibliographystyle{plain}
	\nocite{*}
	\bibliography{biblip}

\begin{thebibliography}{10}

\bibitem{AS1}
Micheal~F. Atiyah and Israel~M. Snger.
\newblock The index of elliptic operators: I.
\newblock {\em Annals of Mathematics}.

\bibitem{BJ}
Saad Baaj and Pierre Julg.
\newblock Th\'eorie bivariante de {K}asparov et op\'erateurs non born\'es dans
  les {$C^{\ast} $}-modules hilbertiens.
\newblock {\em C. R. Acad. Sci. Paris S\'er. I Math.}, 296(21):875--878, 1983.

\bibitem{Bu}
Ulrich Bunke.
\newblock A {$K$}-theoretic relative index theorem and {C}allias-type {D}irac
  operators.
\newblock {\em Math. Ann.}, 303(2):241--279, 1995.

\bibitem{HRk}
Nigel Higson and John Roe.
\newblock {\em Analytic {$K$}-homology}.
\newblock Oxford Mathematical Monographs. Oxford University Press, Oxford,
  2000.
\newblock Oxford Science Publications.

\bibitem{HigRoeI}
Nigel Higson and John Roe.
\newblock Mapping surgery to analysis. {I}. {A}nalytic signatures.
\newblock {\em $K$-Theory}, 33(4):277--299, 2005.

\bibitem{HigRoeII}
Nigel Higson and John Roe.
\newblock Mapping surgery to analysis. {II}. {G}eometric signatures.
\newblock {\em $K$-Theory}, 33(4):301--324, 2005.

\bibitem{HigRoeIII}
Nigel Higson and John Roe.
\newblock Mapping surgery to analysis. {III}. {E}xact sequences.
\newblock {\em $K$-Theory}, 33(4):325--346, 2005.

\bibitem{Hil1}
Michel Hilsum.
\newblock Signature operator on {L}ipschitz manifolds and unbounded {K}asparov
  bimodules.
\newblock In {\em Operator algebras and their connections with topology and
  ergodic theory ({B}u\c steni, 1983)}, volume 1132 of {\em Lecture Notes in
  Math.}, pages 254--288. Springer, Berlin, 1985.

\bibitem{Hil2}
Michel Hilsum.
\newblock Fonctorialit\'e en {$K$}-th\'eorie bivariante pour les vari\'et\'es
  lipschitziennes.
\newblock {\em $K$-Theory}, 3(5):401--440, 1989.

\bibitem{Hil3}
Michel Hilsum.
\newblock L'invariant {$\eta$} pour les vari\'et\'es lipschitziennes.
\newblock {\em J. Differential Geom.}, 55(1):1--41, 2000.

\bibitem{HilSk}
Michel Hilsum and Georges Skandalis.
\newblock Invariance par homotopie de la signature \`a coefficients dans un
  fibr\'e presque plat.
\newblock {\em J. Reine Angew. Math.}, 423:73--99, 1992.

\bibitem{JT}
Kjeld~Knudsen Jensen and Klaus Thomsen.
\newblock {\em Elements of {$KK$}-theory}.
\newblock Mathematics: Theory \& Applications. Birkh\"auser Boston Inc.,
  Boston, MA, 1991.

\bibitem{Ka}
Max Karoubi.
\newblock {$K$}-theory, an elementary introduction.
\newblock In {\em Cohomology of groups and algebraic {$K$}-theory}, volume~12
  of {\em Adv. Lect. Math. (ALM)}, pages 197--215. Int. Press, Somerville, MA,
  2010.

\bibitem{kasp}
G.~G. Kasparov.
\newblock Equivariant {$KK$}-theory and the {N}ovikov conjecture.
\newblock {\em Invent. Math.}, 91(1):147--201, 1988.

\bibitem{Lance}
E.~C. Lance.
\newblock {\em Hilbert {$C^*$}-modules}, volume 210 of {\em London Mathematical
  Society Lecture Note Series}.
\newblock Cambridge University Press, Cambridge, 1995.
\newblock A toolkit for operator algebraists.

\bibitem{PS3}
Paolo Piazza and Thomas Schick.
\newblock Bordism, rho-invariants and the {B}aum-{C}onnes conjecture.
\newblock {\em J. Noncommut. Geom.}, 1(1):27--111, 2007.

\bibitem{PS2}
Paolo Piazza and Thomas Schick.
\newblock The surgery exact sequence, {K}-theory and the signature operator.
\newblock {\em arXiv:1309.4370v1}, 2013.

\bibitem{PS}
Paolo Piazza and Thomas Schick.
\newblock Rho-classes, index theory and {S}tolz' positive scalar curvature
  sequence.
\newblock {\em J. Topol.}, 7(4):965--1004, 2014.

\bibitem{Roe}
John Roe.
\newblock {\em Elliptic operators, topology and asymptotic methods}, volume 179
  of {\em Pitman Research Notes in Mathematics Series}.
\newblock Longman Scientific \& Technical, Harlow; copublished in the United
  States with John Wiley \& Sons, Inc., New York, 1988.

\bibitem{RoeC}
John Roe.
\newblock Comparing analytic assembly maps.
\newblock {\em Q. J. Math.}, 53(2):241--248, 2002.

\bibitem{Sieg}
Paul Siegel.
\newblock Homological calculations with analytic structure groups.
\newblock {\em PhD thesis}, 2012.

\bibitem{SkExt}
Georges Skandalis.
\newblock On the strong {E}xt bifunctor.
\newblock {\em
  http://webusers.imj-prg.fr/{\raisebox{-.6ex}{\symbol{"7E}}}georges.skandalis/Publications/StrongExt.pdf}.

\bibitem{Stolz}
Stephan Stolz.
\newblock Positive scalar curvature metrics---existence and classification
  questions.
\newblock In {\em Proceedings of the {I}nternational {C}ongress of
  {M}athematicians, {V}ol.\ 1, 2 ({Z}\"urich, 1994)}, pages 625--636.
  Birkh\"auser, Basel, 1995.

\bibitem{Sul-Tel}
D.~Sullivan and N.~Teleman.
\newblock An analytic proof of {N}ovikov's theorem on rational {P}ontrjagin
  classes.
\newblock {\em Inst. Hautes \'Etudes Sci. Publ. Math.}, (58):79--81 (1984),
  1983.

\bibitem{Sull}
Dennis Sullivan.
\newblock Hyperbolic geometry and homeomorphisms.
\newblock In {\em Geometric topology ({P}roc. {G}eorgia {T}opology {C}onf.,
  {A}thens, {G}a., 1977)}, pages 543--555. Academic Press, New York, 1979.

\bibitem{Tel1}
Nicolae Teleman.
\newblock The index of signature operators on {L}ipschitz manifolds.
\newblock {\em Inst. Hautes \'Etudes Sci. Publ. Math.}, (58):39--78 (1984),
  1983.

\bibitem{Tel2}
Nicolae Teleman.
\newblock The index theorem for topological manifolds.
\newblock {\em Acta Math.}, 153(1-2):117--152, 1984.

\bibitem{JLT1}
Jean~Louis Tu.
\newblock La conjecture de {N}ovikov pour les feuilletages hyperboliques.
\newblock {\em $K$-Theory}, 16(2):129--184, 1999.

\bibitem{JLT2}
Jean-Louis Tu.
\newblock The gamma element for groups which admit a uniform embedding into
  {H}ilbert space.
\newblock In {\em Recent advances in operator theory, operator algebras, and
  their applications}, volume 153 of {\em Oper. Theory Adv. Appl.}, pages
  271--286. Birkh\"auser, Basel, 2005.

\bibitem{TV}
P.~Tukia and J.~V{\"a}is{\"a}l{\"a}.
\newblock Lipschitz and quasiconformal approximation and extension.
\newblock {\em Ann. Acad. Sci. Fenn. Ser. A I Math.}, 6(2):303--342 (1982),
  1981.

\bibitem{Wprod}
Charlotte Wahl.
\newblock Product formula for {A}tiyah-{P}atodi-{S}inger index classes and
  higher signatures.
\newblock {\em J. K-Theory}, 6(2):285--337, 2010.

\bibitem{Wahl}
Charlotte Wahl.
\newblock Higher {$\rho$}-invariants and the surgery structure set.
\newblock {\em J. Topol.}, 6(1):154--192, 2013.

\bibitem{WY}
Shmuel Weinberger and Guoliang Yu.
\newblock Finite part of operator {$K$}-theory for groups finitely embeddable
  into {H}ilbert space and the degree of nonrigidity of manifolds.
\newblock {\em Geom. Topol.}, 19(5):2767--2799, 2015.

\bibitem{XY2}
Zhizhang Xie and Guoliang Yu.
\newblock Higher rho invariants and the moduli space of positive scalar
  curvature metrics.
\newblock {\em arXiv:1310.1136}.

\bibitem{XY}
Zhizhang Xie and Guoliang Yu.
\newblock Positive scalar curvature, higher rho invariants and localization
  algebras.
\newblock {\em Adv. Math.}, 262:823--866, 2014.

\end{thebibliography}
	
\end{document}